\documentclass[12pt,twoside,english,reqno]{amsart}
\usepackage[T1]{fontenc}
\usepackage[latin9]{inputenc}
\usepackage{geometry}
\usepackage{xcolor}
\usepackage{verbatim}
\usepackage{amstext}
\usepackage{amsthm}
\usepackage{amssymb}
\PassOptionsToPackage{normalem}{ulem}
\usepackage{ulem}
\usepackage{enumitem}

\theoremstyle{plain}
\newtheorem{thm}{\protect\theoremname}[section]
  \theoremstyle{remark}
  \newtheorem{rem}[thm]{\protect\remarkname}

\numberwithin{equation}{section}
\numberwithin{figure}{section}
  \theoremstyle{plain}
  \newtheorem{lem}[thm]{Lemma}
\theoremstyle{example}
  \newtheorem{exam}[thm]{Example}
  \theoremstyle{plain}
  
  \theoremstyle{plain}
  \newtheorem{cor}[thm]{Corollary}
  \newcounter{casectr}

\def\sqr#1#2{{\,\vcenter{\vbox{\hrule height.#2pt\hbox{\vrule width.#2pt
height#1pt \kern#1pt\vrule width.#2pt}\hrule height.#2pt}}\,}}
\def\bo{\sqr44\,}
\def\q{\quad}

\def\wto{\overset{w}{\to}}
\newcommand{\vp}{\varepsilon}
\usepackage{babel}
  \providecommand{\remarkname}{Remark}
 
\providecommand{\theoremname}{Theorem}

\begin{document}

\title[C-H. Chu, M. Rigby]{}

\centerline{}

\noindent {\Large \bf Horoballs and iteration of holomorphic maps\\
on bounded symmetric domains}\\

\author[C-H. Chu, M. Rigby]{}
\noindent{\large Cho-Ho Chu}$\,^{a,*}$, {\large Michael Rigby}$\,^b$

\thanks{* Corresponding author.}
\thanks{{\it E-mail addresses\,}: c.chu@qmul.ac.uk (C-H. Chu), m.rigby@qmul.ac.uk (M. Rigby)}

\maketitle
\vspace{-.3in}

\noindent $^a$ {\footnotesize \it School of Mathematical Sciences, Queen Mary, University of London, London E1 4NS, UK}\\
$^b$ {\footnotesize \it School of Mathematical Sciences, Queen Mary, University of London, London E1 4NS, UK}

\begin{abstract}
Given a fixed-point free compact holomorphic self-map $f$ on a bounded symmetric domain $D$, which may
be infinite dimensional, we establish the existence of a family $\{H(\xi, \lambda)\}_{\lambda >0}$
of convex $f$-invariant domains at a  point $\xi$ in the  boundary  $\partial D$ of $D$, which generalises completely Wolff's theorem for the open unit disc in $\mathbb{C}$. Further, we construct horoballs at $\xi$ and show that they are exactly the $f$-invariant domains when $D$ is of finite rank. Consequently, we show in the latter case that the limit functions of the iterates
$(f^n)$ with weakly closed range all accumulate in one single boundary component of $\partial D$.\\
\end{abstract}


\section{Introduction}

The invariant domains and iteration of a holomorphic self-map on a one-dimensional
bounded symmetric domain is well-understood. In particular, given a fixed-point
free holomorphic self-map $f$ on the unit disc $\mathbb{D}= \{z\in \mathbb{C}:
|z|<1\}$, the celebrated Wolff's theorem \cite{w1}, which can be viewed as an analogue
 of the Schwarz lemma, states that there is a point $\xi$ in
the boundary $\partial \mathbb{D}$ and a family $\{H(\xi, \lambda)\}_{\lambda >0}$ of $f$-invariant domains,
 which covers $\mathbb{D}$ and consists of Euclidean open discs in $\mathbb{D}$ with closure internally tangent to
 $\xi$, in other words, they  are horodiscs of horocentre $\xi$.
 An immediate consequence is the Denjoy-Wolff Theorem \cite{d,w} which asserts the convergence of
the iterates $(f^n)$ to the constant function $h(\cdot) = \xi$. We refer to \cite{bu} for a succinct exposition
of the details and historical remarks.

Although both theorems have been extended completely to Euclidean balls \cite{h} (see also \cite{m}) and various forms of generalisation to other domains in higher dimension have been shown by several authors (e.g.\,\cite{a,bud,chen,cm1,g,kk,m}), a unified treatment for bounded
symmetric domains of all dimensions and a description of the invariant domains resembling Wolff's horodiscs $\{H(\xi, \lambda)\}_{\lambda >0}$ seem wanting, apart from some results in \cite{lie,cr,pauline}.
In this paper, we consider all bounded symmetric domains, including the infinite dimensional ones, and adopt an approach using Jordan theory to the question of invariant domains and iteration of holomorphic maps. This enables us to give a complete generalisation of Wolff's theorem to all bounded symmetric domains and a version of the Denjoy-Wolff Theorem for finite-rank domains, which also unifies and improves the results in \cite{lie,cr,pauline}.

Given a fixed-point free {\it compact} holomorphic self-map $f$ on a bounded symmetric domain $D$, we establish in Theorem \ref{g} the existence of convex $f$-invariant domains $\{H(\xi, \lambda)\}_{\lambda >0}$ at some boundary point $\xi$ of $D$.
Further, we construct the {\it horoballs} at $\xi$, which generalise Wolff's horodiscs, and show in Theorem \ref{interior} and Remark \ref{final} that in the finite-rank case (including  finite dimensions), the invariant domains $H(\xi, \lambda)$ are exactly the horoballs at $\xi$ and affinely homeomorphic to $D$. Using this extension of Wolff's theorem to finite-rank bounded symmetric domains, we show in Theorem \ref{dw} that all limit functions of the iterates $(f^n)$ with weakly closed range must accumulate in one single boundary component in the boundary  $\partial D$.
This generalises the Denjoy-Wolff Theorem for rank-one bounded symmetric domains, which are the Hilbert balls.
For infinite dimensional domains, however,
the result in Theorem \ref{dw} need not be true without the compactness condition on $f$, even for Hilbert balls
\cite{s}. All holomorphic self-maps $f$ on finite dimensional bounded domains are compact.
To achieve these results, we need various topological properties
of finite-rank bounded symmetric domains, which are shown in Section 4 and Section 5, and may be of some independent interest. We conclude with Example \ref{final} to show that, for a M\"obius transformation $g$ of a finite-rank domain $D$, a limit function of the iterates $(g^n)$ can be a constant map or its image is a whole single boundary component of $\partial D$.

A special feature of the paper is the
 substantial use of the underlying Jordan structures of bounded symmetric
domains and some detailed computation involving the Bergmann operators, M\"obius transformations and
Peirce projections. It is possible that such a Jordan approach may also be fruitful in tackling other problems in
complex geometry including the infinite dimensional case.

We begin by explaining briefly the connections between bounded symmetric domains and Jordan theory, but refer to
\cite{book, u} for more details. A bounded symmetric domain is a bounded open connected set $D$ in a complex Banach space such that each point $a\in D$ is an isolated fixed point of an involutive biholomorphic map  $s_a : D \rightarrow D$, called a {\it symmetry} at $a$. The most important connection to Jordan theory for this work is Kaup's Riemann mapping theorem \cite{kaup} asserting that every bounded symmetric domain is biholomorphic to the open unit ball of a JB*-triple $V$, which is a complex Banach space equipped with a Jordan triple structure.

More precisely, a complex Banach space $V$ is called a {\it JB*-triple} if it
admits a continuous triple product
$\{\cdot,\cdot,\cdot\} : V^3 \longrightarrow V$
which is symmetric and linear in the
outer variables, but conjugate linear in the middle variable, and
satisfies
\begin{enumerate}
\item[(i)] (Triple Identity) $\{x,y,\{a,b,c\}\}=\{\{x,y,a\},b,c\}-\{ a,\{y,x,b\},c\}
+\{a,b,\{x,y,c\}\}$; \item[(ii)] $\|\exp it(a \bo a)\|=1$ for all
$t\in \mathbb{R}$; \item[(iii)] $a \bo a$ has
non-negative spectrum; \item[(iv)] $\|a\bo a\|=\|a\|^2$
\end{enumerate}
for $a,b,c,x,y \in V$, where the box operator
$a\bo b: V\rightarrow V$ is defined by $a\bo b(\cdot) =
\{a,b,\cdot\}$.

Open unit balls of JB*-triples are bounded symmetric domains and
in the case of the complex unit disc $\mathbb{D} \subset \mathbb{C}$, the triple product in $\mathbb{C}$ is given by
$\{a,b,c\} = a\overline b c$, where $\overline b$ is the complex conjugate of  $b$. In fact, a Hilbert space
$V$ is a JB*-triple, with triple product $\{a,b,c\} = (\langle a,b\rangle c + \langle c,b\rangle a)/2$, where
$\langle \cdot, \cdot\rangle$ is the inner product. More generally, the Banach space $\mathcal{L}(H,K)$ of bounded linear operators between Hilbert spaces $H$ and $K$, as well as C*-algebras $\mathcal{A}\subset \mathcal{L}(H,H)$,  are JB*-triples with triple product
$$\{a,b,c\} = \frac{1}{2}(ab^*c +cb^*a) \qquad (a,b,c \in \mathcal{L}(H,K) ~ {\rm or} ~ \mathcal{A})$$
where $b^*$ denotes the adjoint operator of $b$.

Throughout the paper, $D$ will always denote a bounded symmetric domain realised as the open unit
ball of a JB*-triple $V$, with boundary $\partial D = \overline D
\backslash D= \{ v\in V: \|v\|=1\}$.  Besides the box operator $a \bo b$ defined above, which satisfies $\|a\bo b\| \leq \|a\| \|b\|$, a
fundamental operator on a JB*-triple $V$ is the {\it Bergmann operator} $B(a,b): V \rightarrow V$ defined by
$$B(a,b)(x) = x - 2\{a,b,x\} + \{a,\{b,x,b\},a\} \qquad (a,b, x \in V)$$
which is invertible if $\|a \bo b\| <1$. We note that
$$\|B(a,b)(x)\| \leq \|x\| + 2\|a\|\|b\|\|x\| + \|a\|^2\|b\|^2\|x\| = (1+\|a\|\|b\|)^2\|x\|.$$
For $a \in D$, the operator $B(a,a)$ has non-negative spectrum
and hence the square roots
$B(a,a)^{\pm 1/2}$ exist and moreover, we have the useful identity
\begin{equation}\label{baa}
\|B(a,a)^{-1/2}\| = \frac{1}{1-\|a\|^2}
\end{equation}
(cf. \cite[Proposition 3.2.13]{book}).
The Bergmann operator $B(a,b)$ on $\mathcal{L}(H,H)$ can be written as
\begin{equation}\label{baa}
B(a,b)(x) = (\mathbf{1}-ab^*)x(\mathbf{1}-b^*a)
\end{equation}
where $\mathbf{1}$ denotes the identity operator on $H$
\cite[p.\,191]{book}. We note that $\|\mathbf{1}-a^*a\| \leq 1$ for  $\|a\| \leq 1$.
On the complex plane $\mathbb{C}$, we have $B(a,b)(x)= (1-a\overline b)^2 x$.
Wolff's horodisc $H(\xi, \lambda)$, which is a one-dimensional horoball with horocentre $\xi$, has the form
$$ H(\xi, \lambda)= t_\lambda^2 \xi + (1-t_\lambda^2)\mathbb{D},\, \qquad t_\lambda^2 =\lambda/(1+\lambda)$$
and in terms of the Bergmann operator, it takes the form
$$H(\xi, \lambda) = t_\lambda^2 \xi + B(t_\lambda \xi, t_\lambda \xi)^{1/2} \mathbb{D}.$$
It is the latter form of the horodisc $H(\xi, \lambda)$ which will be generalised to all bounded symmetric domains of finite rank.

\section{Invariant domains in bounded symmetric domains}

To explain further the underlying idea in our construction of invariant domains and horoballs, we begin with a fixed-point free holomorphic self-map $f$ on the unit disc $\mathbb{D}$ in the complex
plane. The key in Wolff's theorem is to produce a sequence $(z_k)$ in $\mathbb{D}$ converging to
a boundary point $\xi$, which is used to construct the $f$-invariant domains in the following way. For each $\lambda >0$, let
$D_k(\lambda)$ be a Poincar\'e disc (for sufficiently large $k$), which is the open disc centred at $z_k$, with radius $\tanh^{-1}\, r_k$,  measured by the Poincar\'e distance $\kappa\,$:
$$D_k(\lambda) = \{ z\in \mathbb{D}: \kappa (z,z_k) < \tanh^{-1} r_k\}$$
where $r_k \in (0,1)$ satisfies $1-r_k^2=\lambda(1-|z_k|^2)$.
The Poincar\'e disc is the image $g_{z_k} (\mathbb{D}(0,r_k))$ of the Euclidean disc
$\mathbb{D}(0,r_k)= \{z\in \mathbb{D}: |z| < r_k\}$, under the M\"obius transformation
$$g_{z_k}(z) = \frac{z+z_k}{1+ z\overline z_k} \qquad (z \in \mathbb{D}).$$
The $f$-invariant domain $H(\xi,\lambda)$ is then (the interior of) the {\it `limit'} of
the Poincar\'e discs $D_k(\lambda)$ and is given by
\begin{equation}\label{horo1}
H(\xi,\lambda) = \left\{ z\in \mathbb{D}: \frac{|1-z\overline \xi|^2}{1-|z|^2} < \frac{1}{\lambda} \right\}= \frac{\lambda}{1+\lambda} \xi + \frac{1}{1+\lambda} \mathbb{D}
\end{equation}
which is a disc centred at $\frac{\lambda}{1+\lambda} \xi$ with radius $\frac{1}{1+\lambda}$ and as noted earlier,
its boundary is a horocycle with horocentre $\xi$.
 A crucial observation is that this invariant domain is identical with the sets
\begin{eqnarray}\label{horo2}
H(\xi,\lambda)
= \left\{z\in \mathbb{D}: \lim_k \frac{|1-z\overline z_k |^2}{1-|z|^2} < \frac{1}{\lambda}\right\}
= \left\{z\in \mathbb{D}: \lim_k \frac{1-|z_k |^2}{1-|g_{-z_k}(z)|^2} < \frac{1}{\lambda}\right\}.
\end{eqnarray}

Given a bounded symmetric domain realised as the open unit ball $D$ of a JB*-triple $V$, the M\"obius transformations on $D$ and the Kobayashi distance, which generalises the Poincar\'e distance in $\mathbb{D}$, have an explicit Jordan description and the observation in (\ref{horo2}) enables us to extend Wolff's theorem completely to $D$. We carry this out in this section.

We first need some Jordan tools. Let $D$ be the open unit ball of a JB*-triple $V$ and let $a \in D$. The M\"obius transformation $g_a : D \rightarrow D$, induced by $a$, is a biholomorphic map  given by
$$g_a(z) = a + B(a,a)^{1/2}(\mathbf{1}+ z\bo a)^{-1}(z) \qquad (z \in D)$$ with inverse $g_{-a}$, where
 $\mathbf{1}$ denotes the identity operator on $V$.
The Kobayashi distance $\kappa (x,y)$ between two
points $x$ and $y$ in $D$ can be described in terms of a M\"obius transformation:
$$\kappa (x,y) = \tanh^{-1}\|g_{-x}(y)\| = \tanh^{-1}\|g_{-y}(x)\|.$$ By the Schwarz lemma, a holomorphic self-map $f$ on $D$ satisfies
\begin{equation}\label{contrac}
\|g_{-f(x)}(f(y))\| \leq \|g_{-x}(y)\|
\end{equation}
that is, $f$ is $\kappa$-nonexpansive.
We will often use the following norm estimate
\begin{equation}\label{gab}
\frac{1}{1-\|g_{-z}(a)\|^2} = \|B(a,a)^{-1/2} B(a,z) B(z,z)^{-1/2}\| \qquad (a, z\in D)
\end{equation}
which has been proved in \cite{pauline}
(see also \cite[Lemma 3.2.17]{book}).

The aforementioned sequence $(z_k)$ in Wolff's theorem has a limit point $\xi$ by relative compactness
of $\mathbb{D}$. An infinite dimensional domain $D$ need not be relatively compact and the existence
of limit points is not guaranteed. For this reason,
we consider {\it compact} maps $f : D \rightarrow D \subset V$, which are the ones having relatively compact
image $f(D)$, that is, the closure $\overline{f(D)}$ is compact in $V$. All continuous self-maps on a finite dimensional bounded domain are necessarily compact.

Now let $f$ be a compact fixed-point free
holomorphic self-map on $D$.
Choose an increasing sequence $(\alpha_k)$ in $(0,1)$ with
limit $1$. Then $\alpha_k f$ maps $D$ strictly inside itself and
by the fixed-point theorem of Earle and Hamilton \cite{eh}, we have
$\alpha_kf(z_k) = z_k$ for some $z_k \in D$. Note that $z_k \neq
0$. Since $f(D)$ is relatively compact, we may assume, by choosing
a subsequence if necessary, that $(z_k)$ converges to a point $\xi
\in \overline D$. Since $f$ has no fixed point in $D$, the point
$\xi$ must lie in the boundary $\partial D$.

Generalizing Wolff's one-dimensional horodisc, we now define a {\it horoball at} $\xi$ as (the interior of)
a limit of
Kobayashi balls as follows. Alternative descriptions of horospheres in various domains have been given in \cite{a,ba,n}.

Given $\lambda >0$, pick a sequence $(r_k)$ in $(0,1)$ such that
$$ 1- r_k^2= \lambda (1-\|z_k\|^2) $$
from some $k$ onwards. For each $r_k$, define a Kobayashi ball, centred at $z_k$, by
\begin{equation}\label{kball}
D_k(\lambda) =\{z\in D: \kappa (z,z_k) < \tanh^{-1} r_k\}= \{x\in D: \|g_{-z_k}(x)\| < r_k\}
=  g_{z_k}(D(0, r_k))
\end{equation}
where $D(0,r_k)= \{z\in V: \|z\| < r_k\}$. The Kobayashi balls generalise the one-dimensional Poincar\'e discs.

For $x\in \overline{D}$ and $
0<r<1$, using (\ref{kball}) and
 the following formula
\begin{equation}\label{y}
g_z(rx) = (1-r^2)B(rz,rz)^{-1/2}(z) + r
B(z,z)^{1/2}B(rz,rz)^{-1/2}g_{rz}(x),
\end{equation}
as in \cite[Proposition 2.3]{pauline}, one can write
\begin{equation}\label{kball1}
D_k(\lambda) = (1-r_k^2)B(r_k z_k, r_k z_k)^{-1/2} (z_k) + r_k B(z_k.z_k)^{1/2}B(r_k z_k, r_k z_k)^{-1/2}
(D).
\end{equation}
In particular, $D_k(\lambda)$ is a convex domain.

Now we define the {\it limit} of the sequence $(D_k(\lambda))_k$ of Kobayashi balls as the following set:
\begin{equation}\label{hball}
S(\xi, \lambda) = \{ x\in \overline D : x= \lim_k x_k,\, x_k \in D_k(\lambda)\}
\end{equation}
which contains $\xi$ since $z_k \in D_k(\lambda)$, and is also convex because
each $D_k(\lambda)$ is. We call $S(\xi, \lambda)$ a {\it closed horoball ball} at $\xi$, its interior
 $S_0(\xi,\lambda)$  a {\it horoball} at $\xi$ and will
show that, for finite rank $D$, the latter resembles Wolff's  horodisc
and indeed, $S_0(\xi, \lambda) = H(\xi, \lambda)$ in (\ref{horo1}) for $D=\mathbb{D}$. Since $D$ is the interior of
$\overline D$, we have $S_0(\xi, \lambda) \subset D$.

We first construct invariant domains
using the observation in (\ref{horo2}).

\begin{lem}\label{f} The function $F:D\rightarrow [0, \infty)$ given by
\[
F(x)=\limsup_{k\rightarrow\infty}\frac{1-\|z_{k}\|^{2}}{1-\|g_{-z_{k}}(x)\|^{2}}\qquad(x\in D)
\]
is well-defined and continuous. \end{lem} \begin{proof} For each
$x\in D$, we have
\begin{eqnarray*}
\frac{1-\|z_{k}\|^{2}}{1-\|g_{-z_{k}}(x)\|^{2}} & = & \|B(x,x)^{-1/2}B(x,z_{k})B(z_{k},z_{k})^{-1/2}\|(1-\|z_{k}\|^{2})\\
 & \leq & \|B(x,x)^{-1/2}B(x,z_{k})\|\|B(z_{k},z_{k})^{-1/2}\|(1-\|z_{k}\|^{2})\\
 & = & \|B(x,x)^{-1/2}B(x,z_{k})\|\leq\frac{(1+\|x\|\|z_{k}\|)^{2}}{1-\|x\|^{2}}\leq\frac{1+\|x\|}{1-\|x\|}
\end{eqnarray*}
and hence the defining sequence for $F(x)$ is bounded. Therefore
$F$ is well-defined.

For continuity, let $x,y\in D$ and write ${\displaystyle x_{k}=\frac{1-\|z_{k}\|^{2}}{1-\|g_{-z_{k}}(x)\|^{2}}}$,
also ${\displaystyle y_{k}=\frac{1-\|z_{k}\|^{2}}{1-\|g_{-z_{k}}(y)\|^{2}}}$.
Then we have
\[
|x_{k}-y_{k}|\leq\|B(x,x)^{-1/2}B(x,z_{k})-B(y,y)^{-1/2}B(y,z_{k})\|
\]
which gives
\begin{eqnarray*}
|F(x)-F(y)| & = & |\limsup_{k}x_{k}-\limsup_{k}y_{k}|\\
 & \leq & \limsup_{k}\|B(x,x)^{-1/2}B(x,z_{k})-B(y,y)^{-1/2}B(y,z_{k})\|\\
 & = & \|B(x,x)^{-1/2}B(x,\xi)-B(y,y)^{-1/2}B(y,\xi)\|
\end{eqnarray*}
since $(z_{k})$ norm converges to $\xi$. Now continuity of $F$
follows from that of the function\\
 $h(x)=\|B(x,x)^{-1/2}B(x,\xi)\|$ on $D$.

\end{proof}

\begin{rem}\label{nonempty}
We will show in Corollary \ref{h} that $F^{-1}[0,r) \neq \emptyset$ for each $r >0$.
\end{rem}

A similar computation as in the previous proof yields the following
result.

\begin{lem}\label{k} Let $(x_{k})$ be a sequence in $D$ norm converging
to $x\in D$. Then we have
\[
\limsup_{k\rightarrow\infty}\frac{1-\|z_{k}\|^{2}}{1-\|g_{-z_{k}}(x_{k})\|^{2}}=\limsup_{k\rightarrow\infty}\frac{1-\|z_{k}\|^{2}}{1-\|g_{-z_{k}}(x)\|^{2}}.
\]
\end{lem}

The following result is a first generalisation of Wolff's theorem to all bounded symmetric domains.

\begin{thm} \label{g} Let $f$ be a fixed-point free compact holomorphic self-map on a bounded symmetric domain
$D$.
Then there is a sequence $(z_k)$ in $D$ converging
to a boundary point $\xi \in \overline D$ such that, for each $\lambda >0$, the set
$$H(\xi, \lambda)=\left\{x\in D: \limsup_{k\rightarrow \infty} \frac{1-\|z_k\|^2}{1-\|g_{-z_k}(x)\|^2}
 < \frac{1}{\lambda}\right\}$$
is a non-empty convex domain and $f$-invariant, that is, $f(H(\xi,\lambda)) \subset H(\xi, \lambda)$. Moreover,
$D= \bigcup_{\lambda >0} H(\xi, \lambda)$ and $0 \in \bigcap_{\lambda <1} H(\xi, \lambda)$.
\end{thm}
\begin{proof} By Lemma \ref{f}, the set $H(\xi, \lambda)$ is open and by Remark \ref{nonempty}, it is non-empty.
To see that $H(\xi, \lambda)$ is convex, let $x,y\in H(\xi, \lambda)$ and $0< \alpha <1$. We show
$\alpha x + (1-\alpha)y \in H(\xi, \lambda)$. There exists $k_0$ such that
$k \geq k_0$ implies
$$\frac{1-\|z_k\|^2}{1-\|g_{-z_k}(x)\|^2}, ~ \frac{1-\|z_k\|^2}{1-\|g_{-z_k}(y)\|^2}
\, < \,\frac{1}{\beta}  < \frac{1}{\lambda}$$ for some $\beta$ satisfying $\beta (1-\|z_k\|^2)= 1-s_k^2$
with $s_k \in (0,1)$. This gives
 $\|g_{-z_k}(x)\| <s_k$ and  $\|g_{-z_k}(y)\| < s_k$,
that is,
$x$ and $y$ belong to the Kobayashi ball  $D_k(\beta)$ as defined in (\ref{kball}). Since $D_k(\beta)$ is
convex, the element $w=\alpha x + (1-\alpha)y$ is in $D_k(\beta)$ for $k \geq k_0$. Therefore
$$\limsup_{k\rightarrow \infty}\frac{1-\|z_k\|^2}{1-\|g_{-z_k}(w)\|^2} \leq
\limsup_{k\rightarrow \infty}\frac{1-\|z_k\|^2}{1-s_k^2} = \frac{1}{\beta}
 < \frac{1}{\lambda}$$
 and $w \in H(\xi, \lambda)$.

For $f$-invariance, let
$x\in H(\xi, \lambda)$. We need to show $f(x) \in H(\xi, \lambda)$. Let $f_k = \alpha_k f$ be as before. Then by (\ref{contrac}), we have
$$\|g_{-z_k}(f_k(x))\| =\|g_{-f_k(z_k)}(f_k(x)) \leq \|g_{-z_k}(x)\|$$
and
$$\frac{1-\|z_k\|^2}{1-\|g_{-z_k}(f_k(x))\|^2} \leq \frac{1-\|z_k\|^2}{1-\|g_{-z_k}(x)\|^2}.$$
Hence Lemma \ref{k} implies
$$ \limsup_{k\rightarrow \infty} \frac{1-\|z_k\|^2}{1-\|g_{-z_k}(f(x))\|^2}
= \limsup_{k\rightarrow \infty}  \frac{1-\|z_k\|^2}{1-\|g_{-z_k}(f_k(x))\|^2}
\leq \limsup_{k\rightarrow \infty}  \frac{1-\|z_k\|^2}{1-\|g_{-z_k}(x)\|^2} < \frac{1}{\lambda}$$
which gives $f(x) \in H(\xi, \lambda)$.

Finally, for each $y\in D$, we have $y \in H(\xi, \lambda)$ whenever $F(y) < 1/\lambda$. Since $F(0)=1$, we
have $0\in H(\xi, \lambda)$ for $\lambda <1$. This proves the last assertion.
\end{proof}

We will show that the domain $H(\xi, \lambda)$ resembles Wolff's horodisc for finite rank $D$, which
is a second generalisation of Wolff's theorem. This is achieved by showing that $H(\xi, \lambda)$ is identical with
the horoball $S_0(\xi, \lambda)$ and by giving an explicit description of $S_0(\xi, \lambda)$. We first point out
that the result in \cite[Lemma 4.2]{lie} (see also \cite[Theorem 4.6]{lie}) for Lie balls is valid for all bounded symmetric domains, which shows
the invariance of the larger convex set $S(\xi, \lambda) \cap D$. This is stated below.

\begin{thm} Let $f$ be a fixed-point free compact holomorphic self-map on a bounded symmetric domain.
 Then there is a sequence $(z_k)$ in $D$ converging
to a boundary point $\xi \in \overline D$ such that, for each $\lambda >0$, we have $f(S(\xi, \lambda) \cap D) \subset S(\xi, \lambda) \cap D$.
\end{thm}

The following lemma shows the inclusion $H(\xi, \lambda) \subset S(\xi, \lambda) \cap D$.

\begin{lem}\label{subset} $H(\xi, \lambda) \subset S_0(\xi,\lambda)$ and $S(\xi, \lambda)\cap D
 \subset \{x \in D: F(x) \leq 1/\lambda\}$.
\end{lem}
\begin{proof} Let $x\in H(\xi, \lambda)$. Then we have, from some $k$ onwards,
$$\frac{1-\|z_k\|^2}{1-\|g_{-z_k}(x)\|^2} < \frac{1}{\lambda} = \frac{1-\|z_k\|^2}{1-r_k^2}$$
which implies $\|g_{-z_k}(x)\| < r_k$, that is, $x \in D_k(\lambda)$. Hence $x\in S(\xi, \lambda)$.
We have shown $H(\xi, \lambda) \subset S(\xi, \lambda)$ and therefore $H(\xi, \lambda) \subset S_0(\xi, \lambda)$ since
$H(\xi, \lambda)$ is open.

For the second assertion, let $x \in S(\xi, \lambda)\cap D$ with $x = \lim_k x_k$ and $x_k \in D_k(\lambda)$. Since
$\|g_{-z_k}(x_k)\| < r_k$ for each $k$, we have from Lemma \ref{k} that
$$ \limsup_{k\rightarrow \infty} \frac{1-\|z_k\|^2}{1-\|g_{-z_k}(x)\|^2}
= \limsup_{k\rightarrow \infty} \frac{1-\|z_k\|^2}{1-\|g_{-z_k}(x_k)\|^2}
\leq  \limsup_{k\rightarrow \infty} \frac{1-\|z_k\|^2}{1-r_k^2} = \frac{1}{\lambda}.$$
This proves $S(\xi, \lambda) \cap D \subset \{x: F(x) \leq 1/\lambda\}$.
\end{proof}

\begin{exam}\label{function}\rm
Let $D$ be the open unit ball of the JB*-triple $C(\Omega)$  of complex continuous functions on a compact Hausdorff space $\Omega$, with
Jordan triple product $\{x,y,z\} = x\overline y z$ where $\overline y$ denotes the complex conjugate of the
function $y \in C(\Omega)$. For $a,b \in D$, the Bergmann operator $B(a,b)$ is given by a product of functions:
$$B(a,b)(z) = (\mathbf{1}-a\overline b)^2 z \qquad (z \in C(\Omega))$$ where $\mathbf{1}$ denotes the constant function
with value $1$ (cf.\,\cite[Example 3.2.12]{book}) and we have
$$\|B(a,b)\|= \|(\mathbf{1}-a\overline b)^2 \| = \sup\{ |{1}- a(\omega)\overline b(\omega)|^2: \omega \in \Omega\}.$$
Let $(z_k)$ be a sequence in $D$ converging to $\xi \in \partial D$ as before. Then
\begin{eqnarray*}
\|B(x,x)^{-1/2}B(x,z_{k})B(z_{k},z_{k})^{-1/2}\|({1}-\|z_k\|^2) = \left\|\frac{(\mathbf{1}-x\overline z_k)^2({1}-\|z_k\|^2)}{(\mathbf{1}-|x|^2)(\mathbf{1}-|z_k|^2)}\right\|
\end{eqnarray*}
Since $\left\|\frac{{1}-\|z_k\|^2}{\mathbf{1}-|z_k|^2}\right\| \leq 1$ and the sequence $(\mathbf{1}-x\overline z_k)$
converges to $\mathbf{1}-x\overline \xi$ in $C(\Omega)$, we have
$$\limsup_k \|B(x,x)^{-1/2}B(x,z_{k})B(z_{k},z_{k})^{-1/2}\|({1}-\|z_k\|^2)   =
\limsup_k \left\|\frac{(\mathbf{1}-x\overline \xi)^2}{\mathbf{1}-|x|^2}\frac{{1}-\|z_k\|^2}{\mathbf{1}-|z_k|^2}\right\|$$
and for $\lambda >0$,
$$H(\xi, \lambda)=\left\{x\in D: \limsup_k \left\|\frac{(\mathbf{1}-x\overline \xi)^2}{\mathbf{1}-|x|^2}\left(\frac{{1}-\|z_k\|^2}{\mathbf{1}-|z_k|^2}\right)\right\|
<\frac{1}{\lambda}\right\}$$
which reduces to the horodisc in (\ref{horo1}) if $\Omega$ is a singleton in which case the function
$\frac{1-\|z_k\|^2}{\mathbf{1}-|z_k|^2} = \mathbf{1}$.
\end{exam}

\section{Peirce decompositions}

To study the invariant domains $H(\xi, \lambda)$ and the horoballs $S_0(\xi, \lambda)$ in depth,
we need to develop more tools from the ambient Jordan structures. We begin with the Peirce decompositions
of JB*-triples.
An element $e$ in a JB*-triple $V$ is called a {\it tripotent} if $\{e,e,e\} =e$. For a nonzero tripotent $e$,
we have $\|e\|=1$. A nonzero
tripotent $e$ is called {\it minimal} if $\{e, V, e\} = \mathbb{C}\,e$. Two elements $a, b \in V$
are said to be mutually (triple) {\it orthogonal} if $a \bo b = b \bo a =0$, which is equivalent to
$a\bo b=0$ \cite[Lemma 1.2.32]{book}. In this case, we have $\|a+b\|= \max\{\|a\|, \|b\|\}$
\cite[Corollary 3.1.21]{book}. Evidently, if $e$ and $c$ are mutually orthogonal tripotents, then
$e+c$ is also a tripotent.

A tripotent $e\in V$ induces a {\it Peirce decomposition}
$$V= V_0(e) \oplus V_1(e) \oplus V_2(e)$$
where each $V_k(e)$, called the {\it Peirce k-space}, is an
eigenspace
$$V_k(e) = \left\{z\in V : (e\bo e)(z)=\frac{k}{2}\,z\right\} \qquad
(k=0,1,2)$$ of the operator $e\bo e$, and is the range of the
contractive projection $P_k(e): V \longrightarrow V$ given by
$$P_0(e)= B(e,e); \quad P_1(e)=4(e\bo e - (e\bo e)^2); \quad
P_2(e)= 2(e\bo e)^2 - e\bo e.$$ We call $P_k(e)$ the {\it Peirce
k-projection} and refer to \cite[p.\,32]{book} for more detail.

By \cite[Corollary 1.2.46]{book}, a tripotent $c$ is orthogonal to $e$ if and only if
$c \in V_0(e)$. A tripotent $e$ is called {\it maximal} if $V_0(e) =\{0\}$.

 Let $\{e_{1},\dots,e_{n}\}$ be a family  of mutually orthogonal tripotents in a JB*-triple $V$. For $i,j \in \{0,1, \ldots,n\}$, the {\it joint Peirce space} $V_{ij}$ is defined by
\begin{eqnarray*}
V_{ij}:=  V_{ij}(e_{1},\dots,e_{n})  =\{z\in V\medspace:\medspace2\{e_{k},e_{k},z\}=(\delta_{ik}+\delta_{jk})z\medspace\text{for }k=1,\dots,n\},
\end{eqnarray*}
where $\delta_{ij}$ is the Kronecker delta and $V_{ij}=V_{ji}$.

The decomposition
\begin{eqnarray*}
V & = & \bigoplus_{0\leq i\leq j\leq n}V_{ij}
\end{eqnarray*}
is called a {\it joint Peirce decomposition} (cf.\,\cite{lo}).
 More verbosely,
\begin{align*}
V_{ii}= & V_{2}(e_{i}), & i=1,\dots,n;\\
V_{ij}= & V_{ji}=V_{1}(e_{i})\cap V_{1}(e_{j}), & 1\leq i<j\leq n;\\
V_{i0}= & V_{0i}=V_{1}(e_{i})\cap\bigcap_{j\neq i}V_{0}(e_{j}), & i=1,\dots,n;\\
V_{00}= & V_{0}(e_{1})\cap\dots\cap V_{0}(e_{n}).
\end{align*}
The Peirce multiplication rules
\begin{eqnarray*}
\{V_{ij},V_{jk},V_{k\ell}\} & \subset & V_{i\ell}\q {\rm and} \q V_{ij} \bo V_{pq} =\{0\} \q {\rm for} \q
i,j \notin \{p,q\}
\end{eqnarray*}
hold, where we define $\{A,B,C\} = \{\{a,b,c\}: a\in A, b\in B, c\in C\}$ and $A \bo B = \{a\bo b: a\in A,
b\in B\}$ for $A, B, C \subset V$.
  The contractive projection $P_{ij}(e_{1},\dots,e_{n})$ from $V$ onto $V_{ij}(e_{1},\dots,e_{n})$
is called a {\it joint Peirce projection} which satisfies
\begin{equation}
 P_{ij}(e_{1},\dots,e_{n})(e_{k})= \left\{\begin{array}{ll} 0& (i \neq j)\\
                                                   \delta_{ik} e_k & (i=j).\label{pijek}
                                         \end{array}\right.
\end{equation}
We shall simplify the notation $P_{ij}(e_1,\ldots,e_n)$ to $P_{ij}$ if the tripotents $e_1, \ldots, e_n$
are understood.
For a single tripotent $e\in V$,  we have $ P_{11}(e)= P_2(e)$, $ P_{10}(e)=P_1(e) $ and
$P_{00}(e)= P_0(e)$.

Let $M=\{0,1,\dots,n\}$ and $N\subset\{1,\dots,n\}$. The Peirce $k$-spaces of the tripotent $e_{N}=\sum_{i\in N}e_{i}$
 are given by
\begin{eqnarray}
V_{2}(e_{N}) & = & \bigoplus_{i,j\in N}V_{ij},\label{Peirce-2 of sum}\\
V_{1}(e_{N}) & = & \bigoplus_{\substack{i\in N\\
j\in M\backslash N
}
}V_{ij},\\
V_{0}(e_{N}) & = & \bigoplus_{i,j\in M\backslash N}V_{ij}.\label{eq:Peirce-0 of sum}
\end{eqnarray}

\begin{lem}\label{cor: Peirce alg} Let $e_{1},\dots,e_{r}$ be mutually orthogonal tripotents
in a JB*-triple $V$ and let $J\subset\{1,\dots,r\}$ be non-empty. Then
we have \global\long\def\labelenumi{(\roman{enumi})}

\begin{enumerate}
\item[\rm(i)] $P_{ij}(e_{s}\medspace:\medspace s\in J)=P_{ij}(e_{1},\dots,e_{r})$ \quad
for $i,j\in J$,
\item[\rm (ii)] $P_{0j}(e_{s}\medspace:\medspace s\in J)={\displaystyle \sum_{i\in\{0,1,\dots,r\}\backslash J}}P_{ij}(e_{1},\dots,e_{r})$ \quad
for $j\in J$,
\item[\rm (iii)] $P_{00}(e_{s}\medspace:\medspace s\in J)={\displaystyle \sum_{\tiny{\begin{array}{c}
i\leq j\\
i,j\in\{0,1,\dots,r\}\backslash J
\end{array}}}}P_{ij}(e_{1},\dots,e_{r})$.
\end{enumerate}
\end{lem}

\begin{proof}
By re-ordering the indices, we may assume $J=\{1, \ldots, m-1\}$ for some $m \in
\{2, \ldots, r\}$ and it amounts to proving

\begin{enumerate}
\item $P_{ij}(e_{1},\dots,e_{m-1})=P_{ij}(e_{1},\dots,e_{r})$ \quad for $1\leq i\leq j\leq m-1,$
\item $P_{0j}(e_{1},\dots,e_{m-1})={\displaystyle \sum_{i\in\{0\}\cup M}}P_{ij}(e_{1},\dots,e_{r})$
\quad for $1\leq j\leq m-1,$
\item $P_{00}(e_{1},\dots,e_{m-1})={\displaystyle \sum_{\tiny{\begin{array}{c}
i\leq j\\
i,j\in\{0\}\cup M
\end{array}}}}P_{ij}(e_{1},\dots,e_{r}),$
\end{enumerate}
where $M=\{m,m+1,\dots,r\}$.

 For $1\leq i< j\leq m-1$ in (i), we have
\[
P_{ij}(e_{1},\dots,e_{m-1})=P_{1}(e_{i})P_{1}(e_{j})=P_{ij}(e_{1},\dots,e_{r}).
\]
For $i=j$, we have $P_{ij}(e_1, \ldots,e_{m-1}) = P_2(e_i)P_2(e_j) = P_{ij}(e_1, \ldots, e_r)$.

To show (ii), we use the two families $\{e_{1},\dots,e_{m-1}\}$ and
$\{e_{1},\dots,e_{r}\}$ to decompose $V$. In terms of projections
this gives

\begin{equation}
{\displaystyle \sum_{\tiny\begin{array}{c}
0\leq i\leq k\leq m-1\end{array}}}P_{ik}(e_{1},\dots,e_{m-1})={\displaystyle \sum_{\tiny{\begin{array}{c}
0\leq i\leq k\leq r\end{array}}}}P_{ik}(e_{1},\dots,e_{r}).\label{eq:joint Peirce decomp}
\end{equation}
Fix $j\in\{1,\dots,m-1\}$. Applying the Peirce $1$-projection with
respect to the tripotent $e_{j}$ on both sides of \eqref{eq:joint Peirce decomp} and using (i)
gives

\begin{eqnarray*}
&& P_{0j}(e_{1},\dots,e_{m-1})+\sum_{k\in\{1,,\dots,m-1\}\backslash\{j\}}P_{jk}(e_{1},\dots,e_{m-1})\\ & = & P_{0j}(e_{1},\dots,e_{r})+\sum_{k\in\{1,\dots,r\}\backslash\{j\}}P_{jk}(e_{1},\dots,e_{r})\\
&=& P_{0j}(e_{1},\dots,e_{r}) +
\sum_{k\in\{1,\dots,m-1\}\backslash\{j\}}P_{jk}(e_{1},\dots,e_{m-1})+
\sum_{k\in\{m,m+1,\dots,r\}}P_{jk}(e_{1},\dots,e_{r}).
\end{eqnarray*}
Hence

\begin{eqnarray*}
P_{0j}(e_{1},\dots,e_{m-1}) & = & P_{0j}(e_{1},\dots,e_{r})+\sum_{k\in\{m,m+1,\dots,r\}=M}P_{jk}(e_{1},\dots,e_{r})={\displaystyle \sum_{i\in\{0\}\cup M}}P_{ij}(e_{1},\dots,e_{r}).
\end{eqnarray*}

To see (iii), let $1\leq p<q\leq r$. From the definition of the joint Peirce spaces,
we have
$
V_{pq}\subset V_{0}(e_{s})
$
for every $s\in\{1,\dots,r\}\backslash\{p,q\}$. Therefore
\begin{equation}
\sum_{\tiny\begin{array}{c}
m\leq i<k\leq r\end{array}}P_{1}(e_{i})P_{1}(e_{k})=P_{0}(e_{1})\dots P_{0}(e_{m-1})\sum_{\tiny\begin{array}{c}
m\leq i<k\leq r\end{array}}P_{1}(e_{i})P_{1}(e_{k}).\label{eq:Peirce proj inv}
\end{equation}

Applying $P_{00}(e_{1},\dots,e_{m-1})$ to both sides of \eqref{eq:joint Peirce decomp}
and using \eqref{eq:Peirce proj inv} gives
\begin{eqnarray*}
&& P_{00}(e_{1},\dots,e_{m-1})  =  {\displaystyle P_{00}(e_{1},\dots,e_{m-1})\sum_{\tiny\begin{array}{c}
0\leq i\leq k\leq r\end{array}}}P_{ik}(e_{1},\dots,e_{r})\\
 & = & P_{00}(e_{1},\dots,e_{m-1})\left(P_{00}(e_{1},\dots,e_{r})+\sum_{k=1}^{r}P_{0k}(e_{1},\dots,e_{r})+
 \hspace{-.1in}\sum_{\tiny\begin{array}{c}
0<i<k\leq r\end{array}}\hspace{-.2in}P_{ik}(e_{1},\dots,e_{r})+\sum_{i=1}^{r}P_{ii}(e_{1},\dots,e_{r})\right)\\
 & = & P_{0}(e_{1})\dots P_{0}(e_{m-1})\left(P_{0}(e_{1})\dots P_{0}(e_{r})+\sum_{k=1}^{r}P_{1}(e_{k})\hspace{-.1in}\prod_{\tiny\begin{array}{c}
i\neq k\\
1\leq i\leq r
\end{array}}\hspace{-.2in}P_{0}(e_{i})+ \hspace{-.1in}\sum_{\tiny\begin{array}{c}
0<i<k\leq r\end{array}}\hspace{-.2in}P_{1}(e_{i})P_{1}(e_{k})~+\,\sum_{i=1}^{r}P_{2}(e_{i})\right)\\
 & = & P_{0}(e_{1})\dots P_{0}(e_{r})+\sum_{k=m}^{r}P_{1}(e_{k})\hspace{-.1in}\prod_{\small\begin{array}{c}
i\neq k\\
1\leq i\leq r
\end{array}}\hspace{-.2in}P_{0}(e_{i})\,+\,\sum_{\tiny\begin{array}{c}
m\leq i<k\leq r\end{array}}\hspace{-.2in}P_{1}(e_{i})P_{1}(e_{k})\, +\, \sum_{i=m}^{r}P_{2}(e_{i})\\
 & = & P_{00}(e_{1},\dots,e_{r})+\sum_{k=m}^{r}P_{0k}(e_{1},\dots,e_{r})+\sum_{\tiny\begin{array}{c}
m\leq i<k\leq r\end{array}}P_{ik}(e_{1},\dots,e_{r})+\sum_{i=m}^{r}P_{ii}(e_{1},\dots,e_{r})\\
 & = & {\displaystyle \sum_{\tiny\begin{array}{c}
i\leq j\\
i,j\in\{0\}\cup M
\end{array}}}P_{ij}(e_{1},\dots,e_{r}).
\end{eqnarray*}
\end{proof}
The Peirce projections provide a very useful formulation of the Bergmann
operators.
Let $e_{1},\dots,e_{n}$
be  mutually triple orthogonal tripotents in a JB*-triple $V$ and
let $x=\sum_{i=1}^{n}\lambda_{i}e_{i}$ with
$\lambda_{i}\in\mathbb{C}$.
Then the Bergmann
operator $B(x,x)$ satisfies
\begin{eqnarray}
B(x,x) & = & \sum_{0\leq i\leq j\leq n}(1-|\lambda_{i}|^{2})(1-|\lambda_{j}|^{2})P_{ij}.\label{eq:Loos Bergmann}
\end{eqnarray}
where we set $\lambda_{0}=0$ and $P_{ij}=P_{ij}(e_1, \ldots,e_n)$.  This gives the following formulae
for the square roots
\begin{eqnarray}
B(x,x)^{1/2} & = & \sum_{0\leq i\leq j\leq n}(1-|\lambda_{i}|^{2})^{1/2}(1-|\lambda_{j}|^{2})^{1/2}P_{ij}\quad (\|x\|<1) \label{eq:Bergmann Sq Rt}\\
B(x,x)^{-1/2} & = & \sum_{0\leq i\leq j\leq n}(1-|\lambda_{i}|^{2})^{-1/2}(1-|\lambda_{j}|^{2})^{-1/2}P_{ij} \quad (\|x\|<1).\label{eq:Bergmann Neg Sq Rt}
\end{eqnarray}

\section{Finite-rank bounded symmetric domains}

Finite-rank bounded symmetric domains are (biholomorphically equivalent to) open unit balls of
finite-rank JB*-triples. To describe them, we first recall the definition of the rank of a JB*-triple.
A closed subspace $E$ of a JB*-triple $V$ is called a {\it subtriple} if $a,b,c \in E$ implies
$\{a,b,c\} \in E$. The Peirce spaces $V_{ij}$ defined before are subtriples of $V$. For each $a\in V$, let $V(a)$ be the smallest closed subtriple of $V$ containing $a$. For $V\neq \{0\}$,
the {\it rank} of $V$ is defined to be
$$r(V) = \sup \{ \dim V(a): a\in V\} \in \mathbb{N}\cup \{\infty\}.$$
A (nonzero) JB*-triple $V$ has {\it finite rank}, that is, $r(V) < \infty$ if, and only if, $V$ is a reflexive Banach space \cite[Proposition 3.2]{kaup1}. In this case,
its rank $r(V)$ is the (unique) cardinality of a maximal family of mutually orthogonal minimal tripotents
and $V$
 is an $\ell^{\infty}$-sum of a finite
number of finite-rank Cartan factors, which
can be infinite dimensional. There are six types of finite-rank Cartan factors, listed below.
$$
\begin{aligned}
\text{\rm Type I} &\q \mathcal{L}(\mathbb{C}^r,H)\q (r =1,2, \ldots),\q \\ \text{\rm Type II} &\q
\{z\in \mathcal{L}(\mathbb{C}^r,\mathbb{C}^r): z^t=-z\}\q (r = 5,6,  \ldots),\\ \text{\rm Type III} &\q  \{z\in
\mathcal{L}(\mathbb{C}^r,\mathbb{C}^r): z^t=z\} \q (r = 2,3,  \ldots),\\ \text{\rm Type IV} &\q \text{\rm spin
factor,}\\ \text{\rm Type V} &\q M_{1,2}(\mathcal{O}) = 1\times
2\q\text{\rm matrices over the Cayley algebra}\; \mathcal{O},\\
\text{\rm Type VI} &\q M_3(\mathcal{O})=3\times 3 \q\text{\rm
hermitian matrices over}\; \mathcal{O},
\end{aligned}
$$
where $\mathcal{L}(\mathbb{C}^r,H)$ is the JB*-triple of linear
operators from $\mathbb{C}^r$ to a Hilbert space $H$ and
$z^t$ denotes
the transpose of $z$ in the JB*-triple $\mathcal{L}(\mathbb{C}^r,\mathbb{C}^r) $ of $r \times r$ complex
matrices.   A {\it spin
factor} is a JB*-triple $V$  equipped with a complete inner
product $\langle \cdot,\cdot \rangle$ and a conjugation $* : V \rightarrow V$ satisfying
 $$ \langle x^*, y^*\rangle = \langle y, x\rangle \quad {\rm and} \q  \{x,y,z\} =\frac
12\,\big(\langle x,y\rangle z + \langle z,y\rangle x - \langle x,z^*\rangle y^*\big). $$
The only possible infinite dimensional finite-rank Cartan factors are
the spin factors and $L(\mathbb{C}^{r},H)$, with $\dim H = \infty$, where a spin factor has rank 2 and $L(\mathbb{C}^r,H)$ has rank $r$. The open unit ball of a spin factor is known as a
{\it Lie ball}. The open unit balls of the first four types of finite dimensional Cartan factors are
the classical {\it Cartan domains}.

Let  $V$ be a  JB*-triple of finite rank $r$. Then the sum of $r$ orthogonal minimal
tripotents $e_1, \ldots, e_r$ is a maximal tripotent and $V_{00}=V_{0}(e_{1}+\dots+e_{r})=\{0\}$
by \eqref{eq:Peirce-0 of sum}.
Each $x\in V$ has a {\it spectral decomposition}
\begin{eqnarray}
x & = & \alpha_{1}e_{1}+\dots+\alpha_{r}e_{r}\label{eq:Spec Decomp-2}
\end{eqnarray}
for some mutually orthogonal minimal tripotents $e_{1},\dots,e_{r}$,
where the uniquely determined coefficients satisfy
 $0\leq\alpha_{r}\leq\dots\leq\alpha_{1}$ with
$\alpha_{1}=\|x\|$.

Since a finite-rank JB*-triple is  reflexive, its open unit ball is
relatively compact in the weak topology. We will exploit the weak topology in our
computation in the infinite dimensional case, which involves the spin factors
and the Type I Cartan factors $L(\mathbb{C}^r, H)$.

The minimal tripotents in $L(\mathbb{C}^{r},H)$ are exactly
the rank-one operators $a\otimes b:\mathbb{C}^{r}\to H$ with $\|a\|_{\mathbb{C}^{r}}=\|b\|_{H}=1$, where
\[
(a\otimes b)(\mu)=\langle \mu,a\rangle b\quad(\mu\in\mathbb{C}^{r})
\]
and the adjoint
$(a\otimes b)^{\ast}$ is the rank-one operator $b\otimes a:H\to\mathbb{C}^{r}$ given
by $(b\otimes a)(h)=\langle h,b\rangle a$ for $h\in H$. We have used the same symbol $\langle
\cdot, \cdot\rangle$ for inner products in $\mathbb{C}^r$ and $H$, which should not cause any
confusion.
For convenience, we write $x\perp y$ to denote that $x$ and $y$
are orthogonal in a Hilbert space, not to be confused with the notion of
orthogonality in a JB*-triple.

\begin{lem} \label{lem:Rank-one operators triple orth}Let $e_{i}=a_{i}\otimes b_{i}$
be a rank-one operator in $L(\mathbb{C}^{r},H)$ for $i=1,2$.\emph{
Then $e_{1} \bo e_{2} = 0$ if and only if $a_{1}\perp a_{2}$
and $b_{1}\perp b_{2}$ in their respective Hilbert spaces. In this
case, we have }$e_{1}(\mu)\perp e_{2}(\mu)$ in $H$ \emph{for all
$\mu\in\mathbb{C}^{r}$. }\end{lem}

\begin{proof}Let $e_{1} \bo e_{2}= 0$.
Then we have $0=\{e_{1},e_{2},f\}=\frac{1}{2}(e_{1}e_{2}^{\ast}f+fe_{2}^{\ast}e_{1})$
for all $f\in L(\mathbb{C}^{r},H)$. More explicitly we have
\begin{eqnarray*}
0 & = & \langle\langle f(\cdot),b_{2}\rangle a_{2},a_{1}\rangle b_{1}+f[\langle\langle\cdot,a_{1}\rangle b_{1},b_{2}\rangle a_{2}]\\
 & = & \langle f(\cdot),b_{2}\rangle\langle a_{2},a_{1}\rangle b_{1}+\langle\cdot,a_{1}\rangle\langle b_{1},b_{2}\rangle f(a_{2}).
\end{eqnarray*}
In particular, when $f=e_{1}$, we have $0=\langle\cdot,a_{1}\rangle\langle b_{1},b_{2}\rangle\langle a_{2},a_{1}\rangle b_{1}$,
which clearly implies either $a_{1}\perp a_{2}$ or $b_{1}\perp b_{2}$.
On the other hand, if $f=e_{2}$, then
\begin{eqnarray*}
0 & = & \langle\langle\cdot,a_{2}\rangle b_{2},b_{2}\rangle\langle a_{2},a_{1}\rangle b_{1}+\langle\cdot,a_{1}\rangle\langle b_{1},b_{2}\rangle\langle a_{2},a_{2}\rangle b_{2}\\
 & = & \langle\cdot,a_{2}\rangle\|b_{2}\|^{2}\langle a_{2},a_{1}\rangle b_{1}+\langle\cdot,a_{1}\rangle\langle b_{1},b_{2}\rangle\|a_{2}\|^{2}b_{2}.
\end{eqnarray*}
This gives  $a_{1}\perp a_{2}$ and $b_{1}\perp b_{2}$. In
this case, we also have $\langle e_{1}(\mu),e_{2}(\mu)\rangle=\langle\langle\mu,a_{1}\rangle b_{1},\langle\mu,a_{2}\rangle b_{2}\rangle=\langle\mu,a_{1}\rangle\overline{\langle\mu,a_{2}\rangle}\langle b_{1},b_{2}\rangle=0$,
for $\mu\in\mathbb{C}^{r}$.

Conversely, suppose $a_{1}\perp a_{2}$ and $b_{1}\perp b_{2}$.
Given any $x\in L(\mathbb{C}^{r},H)$, we have
\begin{eqnarray*}
2(e_{1}\bo e_{2})(x) & = & 2\{e_{1},e_{2},x\}
  =  e_{1}e_{2}^{\ast}x+xe_{2}^{\ast}e_{1}\\
 & = & \langle x(\cdot),b_{2}\rangle\langle a_{2},a_{1}\rangle b_{1}+\langle\cdot,a_{1}\rangle\langle b_{1},b_{2}\rangle x(a_{2})
  =  0.
\end{eqnarray*}
\end{proof}

Let $D$ be a Lie ball, realised as the open unit ball
 of a spin factor $V$, equipped with an inner product $\langle \cdot,\cdot\rangle$
 and involution $*$.
We now prove a convergence result for $D$ in the following lemma which will be used later. This result
also simplifies
and improves the arguments showing the convergence of the sequence $(c_k(y))$ in \cite[p.\,130-132]{lie}.

To prove the lemma,
we make use of the fact \cite[Lemma 5.10]{js} that for any two triple orthogonal elements $u,v\in {D}$, the M\"obius transformation
$g_{u+v}$ satisfies
\[
g_{u+v}=g_{u}\circ g_{v}.
\]

\begin{lem}\label{d} Let $(z_{k})$ be a sequence in a Lie ball $D$ norm converging
to $\xi\in\overline{D}$, where
\[
z_{k}=\alpha_{1k}d_{k}+\alpha_{2k}d_{k}^{*},
\]
$(d_k)$ is a sequence of minimal tripotents weakly converging to $d\in\overline{D}$ and
$\alpha_{1k}=\|z_k\|\geq |\alpha_{2k}|$ with
$\alpha_{2k} \in \mathbb{C}$.

If $\xi\neq0$, then both sequences $(d_{k})$ and $(d_{k}^{*})$
are norm convergent, in which case $d$ and $d^*$ are minimal tripotents.
If $\xi=0$, then both sequences $(\{d_{k},z_{k},d_{k}\})$ and $(\{d_{k}^{*},z_{k},d_{k}^{*}\})$
norm converge to $0$. \end{lem}

\begin{proof}

We note that $d_k \bo d_k^* =\langle d_k,d_k^*\rangle =0$ \cite[Lemma 2.3]{lie} and $\lim_{k}\alpha_{1k}=\|\xi\|$. Also, $\lim_{k}\alpha_{2k}=\alpha\in\overline{\mathbb{D}}$
with $|\alpha|\leq\|\xi\|$.

If $\xi=0$, then $\lim_{k}\alpha_{1k}=0$ and $\alpha=0$. Hence
both sequences $\{d_{k},z_{k},d_{k}\}=\alpha_{1k}d_{k}$ and $\{d_{k}^{*},z_{k},d_{k}^{*}\}=\overline{\alpha_{2k}}d_{k}^{*}$
norm converge to $0$.

Let $\xi\neq0$. We show every subsequence of $(d_{k})$ contains
a norm convergent subsequence which would complete the proof. To simplify
notation, we pick a subsequence and still denote it by $(d_{k})$.

We first consider the case $\|\xi\|=1$ in which situation, it has
been shown in \cite[p.126]{lie} that $(d_{k})$ is norm convergent
if $|\alpha|<1$. Hence we only need to show norm convergence for
$|\alpha|=1$. In this case, the Bergmann operators $B(z_k,z_k)$ norm
converge to $0$ by the formula
$$B(z_k,z_k) = (1-\alpha_{1k}^2)^2P_2(d_k) + (1-\alpha_{1k}^2)(1-|\alpha_{2k}|^2)
P_1(d_k) + (1-|\alpha_{2k}|^2)^2P_0(d_k)$$
since the Peirce projections are contractive.
Pick $y\in D\backslash\mathbb{C}\xi$ and let $w_k = g_{-z_k}(y)$. Then
\[
\lim_k w_k = \lim_{k}\,( -z_{k} + B(z_k,z_k)^{1/2}(\mathbf{1}- y\bo z_{k})^{-1}(y)) = -\xi.
\]

Write $g_{z_{k}}(w_{k})=g_{\alpha_{1k}d_{k}}\,(g_{\alpha_{2k}d_{k}^{*}}(w_{k}))$
and $x_{k}=g_{\alpha_{2k}d_{k}^{*}}(w_{k})$.\\

For each $x\in D$, we have
\begin{eqnarray*}
(x\bo\alpha_{1k}d_{k})(x) & = & \alpha_{1k}\langle x,d_{k}\rangle x-\frac{\alpha_{1k}\langle x,x^{*}\rangle}{2}\,d_{k}^{*}\\
(x\bo\alpha_{1k}d_{k})^{2}(x) & = & \alpha_{1k}^{2}\langle x,d_{k}\rangle^{2}x-\frac{\alpha_{1k}^{2}\langle x,x^{*}\rangle\langle x,d_{k}\rangle}{2}\,d_{k}^{*}\\
(x\bo\alpha_{1k}d_{k})^{3}(x) & = & \alpha_{1k}^{3}\langle x,d_{k}\rangle^{3}x-\frac{\alpha_{1k}^{3}\langle x,x^{*}\rangle\langle x,d_{k}\rangle^{2}}{2}\,d_{k}^{*}\\
\end{eqnarray*}
and so on. Hence
\begin{eqnarray*}
 &  & (I+x\bo d_{k})^{-1}{x}=(I-x\bo d_{k}+(x\bo d_{k})^{2}-\cdots)(x)\\
 & = & (1-\alpha_{1k}\langle x,d_{k}\rangle+\alpha_{1k}^{2}\langle x,d_{k}\rangle^{2}-\cdots)x+\frac{\alpha_{1k}\langle x,x^{*}\rangle}{2}(1-\alpha_{1k}\langle x,d_{k}\rangle+\alpha_{1k}^{2}\langle x,d_{k}\rangle^{2}-\cdots)d_{k}^{*}\\
 & = & \frac{x}{1+\alpha_{1k}\langle x,d_{k}\rangle}+\frac{\alpha_{1k}\langle x,x^{*}\rangle\,d_{k}^{*}}{2(1+\alpha_{1k}\langle x,d_{k}\rangle)}.
\end{eqnarray*}

It follows from
\[
B(\alpha_{1k}d_{k},\alpha_{1k}d_{k})^{1/2}=P_{0}(d_{k})+\sqrt{1-\alpha_{1k}^{2}}P_{1}(d_{k})+(1-\alpha_{1k}^{2})P_{2}(d_{k})
\]
that
\begin{eqnarray*}
g_{\alpha_{1k}d_{k}}(x) & = & \alpha_{1k}d_{k}+\frac{2\langle x,d_{k}^{*}\rangle+\alpha_{1k}\langle x,x^{*}\rangle}{2(1+\alpha_{1k}\langle x,d_{k}\rangle)}\,d_{k}^{*}\\
 & + & \sqrt{1-\alpha_{1k}^{2}}P_{1}(d_{k})\left(\frac{x}{1+\alpha_{1k}\langle x,d_{k}\rangle}\right)+(1-\alpha_{1k}^{2})P_{2}(d_{k})\left(\frac{x}{1+\alpha_{1k}\langle x,d_{k}\rangle}\right).
\end{eqnarray*}
Likewise, we have
\begin{eqnarray*}
g_{\alpha_{2k}d_{k}^{*}}(x) & = & \alpha_{2k}d_{k}^{*}+\frac{2\langle x,d_{k}\rangle+\overline{\alpha}_{2k}\langle x,x^{*}\rangle}{2(1+\overline{\alpha}_{2k}\langle x,d_{k}^{*}\rangle)}\,d_{k}\\
 & + & \sqrt{1-|\alpha_{2k}|^{2}}P_{1}(d_{k})\left(\frac{x}{1+\overline{\alpha}_{2k}\langle x,d_{k}^{*}\rangle}\right)+(1-|\alpha_{2k}|^{2})P_{0}(d_{k})\left(\frac{x}{1+\overline{\alpha}_{2k}\langle x,d_{k}^{*}\rangle}\right).
\end{eqnarray*}
Therefore we have
\begin{eqnarray*}
g_{\alpha_{1k}d_{k}}(g_{\alpha_{2k}d_{k}^*}(w_{k})) & = & g_{\alpha_{1k}d_{k}}(x_{k})=\alpha_{1k}d_{k}+\frac{2\langle x_{k},d_{k}^{*}\rangle+\alpha_{1k}\langle x_{k},x_{k}^{*}\rangle}{2(1+\alpha_{1k}\langle x_{k},d_{k}\rangle)}\,d_{k}^{*}\\
 & + & \sqrt{1-\alpha_{1k}^{2}}P_{1}(d_{k})\left(\frac{x_{k}}{1+\alpha_{1k}\langle x_{k},d_{k}\rangle}\right)+(1-\alpha_{1k}^{2})P_{2}(d_{k})\left(\frac{x_{k}}{1+\alpha_{1k}\langle x_{k},d_{k}\rangle}\right).
\end{eqnarray*}

Since
\begin{eqnarray*}
\left|\frac{\sqrt{1-|\alpha_{2k}|^{2}}}{1+\overline{\alpha}_{2k}\langle w_{k},d_{k}^{*}\rangle}\right|^{2} & = & \frac{1-|\alpha_{2k}|^{2}}{1+|\alpha_{2k}|^{2}|\langle w_{k},d_{k}^{*}\rangle|^{2}+2{\rm Re}\,\overline{\alpha}_{2k}\langle w_{k},d_{k}^{*}\rangle}\\
 & = & 1-\frac{|\alpha_{2k}|^{2}+|\alpha_{2k}|^{2}\langle w_{k},d_{k}^{*}\rangle|^{2}+2{\rm Re}\,\overline{\alpha}_{2k}\langle w_{k},d_{k}^{*}\rangle}{1+|\alpha_{2k}|^{2}|\langle w_{k},d_{k}^{*}\rangle|^{2}+2{\rm Re}\,\overline{\alpha}_{2k}\langle w_{k},d_{k}^{*}\rangle}\leq2,
\end{eqnarray*}
we may assume, by choosing a subsequence, that the complex sequence
\[
\frac{\sqrt{1-|\alpha_{2k}|^{2}}}{1+\overline{\alpha}_{2k}\langle w_{k},d_{k}^{*}\rangle}
\]
converges. Likewise, we may assume that the sequence
\[
\frac{\sqrt{1-\alpha_{1k}^{2}}}{1+\alpha_{1k}\langle x_{k},d_{k}\rangle}
\]
converges.

Observe that
\[
P_{1}(d_{k})(x_{k})=P_{1}(d_{k})(g_{\alpha_{2k}d_{k}^{*}}(w_{k}))=\frac{\sqrt{1-|\alpha_{2k}|^{2}}}{1+\overline{\alpha}_{2k}\langle w_{k},d_{k}^{*}\rangle}P_{1}(d_{k})(w_{k}).
\]
Writing
\[
P_{1}(d_{k})(w_{k})=w_{k}-\langle w_{k},d_{k}\rangle d_{k}-\langle w_{k},d_{k}^{*}\rangle d_{k}^{*},
\]
it can be seen that
\[
g_{z_k}(w_k)=g_{\alpha_{1k}d_{k}}(g_{\alpha_{2k}d_{k}^*}(w_{k}))
=\frac{\sqrt{1-\alpha_{1k}^{2}}\sqrt{1-|\alpha_{2k}|^{2}}}{(1+\alpha_{1k}\langle x_{k},d_{k}\rangle)(1+\overline{\alpha}_{2k}\langle w_{k},d_{k}^{*}\rangle)}\,w_{k}+A_{k}d_{k}+B_{k}d_{k}^{*}
\]
for some $A_{k},B_{k}\in\mathbb{C}$. From this we infer that the
sequence $(A_{k}d_{k}+B_{k}d_{k}^{*})$ norm converges to $y+\beta\xi$
for some $\beta\in\mathbb{C}$. Combining this norm convergence with the
given norm convergence of the sequence $(\alpha_{1k}d_{k}+\alpha_{2k}d_{k}^{*})$
to $\xi$
and, noting $y\notin\mathbb{C}\xi$, we conclude that the sequence
$(d_{k})$ is norm convergent to $d$.

Finally, consider the case $\|\xi\|<1$. Since $\xi\neq0$, we may
assume $z_{k}\neq0$ by omitting, if necessary, the first few terms
of the sequence. Let
\[
z_{k}'=\frac{z_{k}}{(1+1/2^{k})\|z_{k}\|}=\frac{\alpha_{1k}}{(1+1/2^{k})\|z_{k}\|}d_{k}+\frac{\alpha_{2k}}{(1+1/2^{k})\|z_{k}\|}d_{k}^{*}.
\]
Then $(z_{k}')$ is a sequence in $D$ norm converging to $\xi/\|\xi\|$.
By the previous case, one concludes with the norm convergence $d=\lim_{k}d_{k}$.
This proves the first assertion.

\end{proof}

\section{Horoballs in finite rank bounded symmetric domains}

In this section, we show that the invariant domains $H(\xi,\lambda)$ of a fixed-point free
compact holomorphic self-map $f$ on a finite rank bounded symmetric domain are horoballs $S_0(\xi, \lambda)$
resembling the horodiscs in Wolff's theorem for $\mathbb{D}$.

Throughout, let $D$ be a finite-rank bounded symmetric domain, realised as the open unit
 ball of a JB{*}-triple $V$  of rank $p$, with a decomposition
\[
V=V_{1}\oplus\cdots\oplus V_{q}
\]
into an $\ell^{\infty}$-sum of Cartan factors $V_{1}$, \ldots{},
$V_{q}$, which are mutually orthogonal, that is, $V_i \bo V_j = \{0\}$ for
distinct $i,j \in \{1, \ldots,q\}$. For each $z\in V$ in the sequel, we shall write the spectral
decomposition of $z$ in the following form
\begin{equation}
z=\alpha_{1}e_{1}+\cdots+\alpha_{p}e_{p}\qquad(\alpha_{1},\ldots,\alpha_{p}\geq0)\label{sp1}
\end{equation}
where two consecutive minimal tripotents are either in the same direct
summand, or belong to two consecutive summands, that is, there exist
$i\geq1$ and $1\leq r<\cdots<\ell$ such that
\begin{equation}
\{e_{1},\ldots,e_{i}\}\subset V_{1},~\{e_{i+1},\ldots,e_{i+r}\}\subset V_{2},\cdots,\{e_{i+\ell},\ldots,e_{p}\}\subset V_{q}.\label{sp2}
\end{equation}

Given a direct sum decomposition $z=z_{1}+\cdots+z_{q}\in V_1 \oplus\cdots\oplus V_q$, triple orthogonality
implies
\begin{equation}
\{e_{k},z,e_{k}\}=\{e_{k},z_{j},e_{k}\}\label{eq}
\end{equation}
if $e_{k}$ belongs to the summand $V_{j}$. Considering each element $z\in V$ having
$q$ summands, we note that norm and weak convergence of a sequence
in $V$ are the same as norm and weak convergence in each summand.
In fact, the weak and norm topologies of $V$ are product topologies
of those of $V_{1},\ldots,V_{q}$.

We sometimes use the symbol $x_k \wto x$ to denote
weak convergence of a sequence $(x_k)$.
In the Cartan factor $L(\mathbb{C}^r,H)$, we have $x_k \wto x$ if and only if
$\langle x_k(\mu), h \rangle \rightarrow \langle x(\mu),h\rangle$ as $k \rightarrow \infty$, for all $\mu \in \mathbb{C}^r$ and $h \in H$. The weak convergence $x_k \wto x$ implies $\{x_k,a,x_k\} \wto \{x,a,x\}$ for each
$a \in L(\mathbb{C}^r,H)$ since
$$\langle \{x_k,a,x_k\}(\mu), \,h\rangle = \langle x_k a^* x_k (\mu), \,h\rangle = \langle a^*x_k(\mu), x_k^*(h)\rangle \rightarrow \langle a^*x(\mu), x^*(h)\rangle$$ where weak and norm convergence are the same in $\mathbb{C}^r$.

\begin{lem}\label{fr} Let $V$ be a finite-rank JB{*}-triple without
a spin factor direct summand and $(z_{k})$ a sequence in $V$ norm converging
to some $\xi\in V$. Given a sequence $(e_{k})$ of minimal tripotents
in $V$ weakly converging to $e\in V$, we have the weak convergence
$\{e_{k},z_{k},e_{k}\}\wto\{e,\xi,e\}$ and also, $\{e,V,e\}\subset\mathbb{C}\,e$.
\end{lem}

\begin{proof} Let $V$ be an $\ell^{\infty}$-sum $V_{1}\oplus\cdots\oplus V_{q}$
of mutually orthogonal Cartan factors. If the assertion is true for each $V_{j}$
($j=1,\ldots,q$), then it is also true for $V$ by (\ref{eq}) and
the subsequent remark there. Hence it suffices to prove the lemma
for a Type I Cartan factor $V=L(\mathbb{C}^{r},H)$.

The above remark allows us to assume $\xi=0$ and show $\{e_{k},z_{k},e_{k}\}\wto0$.
Indeed, we have
\begin{eqnarray*}
|\langle e_{k}z_{k}^{\ast}e_{k}(\mu),h\rangle| & = & |\langle z_{k}^{\ast}e_{k}(\mu),e_{k}^{*}h\rangle|\\
 & \leq & \|z_{k}^{\ast}e_{k}(\mu)\|_{\mathbb{C}^{r}}.\|e_{k}^{*}h\|_{\mathbb{C}^{r}}\\
 & \leq & \|z_{k}^{\ast}\|.\|e_{k}(\mu)\|_{H}.\|e_{k}^{*}\|.\|h\|_{H}\\
 & \leq & \|z_{k}\|.\|\mu\|_{\mathbb{C}^{r}}.\|h\|_{H}
  \to  0
\end{eqnarray*}
for all $\mu\in\mathbb{C}^{r}$ and $h\in H$.

For the second assertion, let $z\in V$. Then
$
\{e,z,e\}= \mbox{weak-}\lim_{k}\{e_{k},z,e_{k}\}= \mbox{weak-}\lim_{k}\lambda_{k}e_{k}
$
for some $\lambda_{k}\in\mathbb{C}$. It follows that the sequence
$|\lambda_{k}|=\|\lambda_{k}e_{k}\|$ is bounded and there is a subsequence
of $(\lambda_{k})$ converging to some $\lambda\in\mathbb{C}$. This
gives $\{e,z,e\}=\lambda e$. \end{proof}

\begin{lem}\label{fr'} Let $V$ be a JB{*}-triple of finite rank
$p$ and $(z_{k})$ a sequence in $D$ norm converging to some $\xi\in\overline{D}$,
with spectral decomposition
\[
z_{k}=\alpha_{1k}e_{1k}+\cdots+\alpha_{pk}e_{pk}
\]
where each sequence $(e_{ik})_{k}$ weakly converges to $e_{i}$ for
$i=1,\ldots,p$.
Then the sequence $(\{e_{ik},z_{k},e_{ik}\})_{k}$ weakly converges
to $\{e_i,\xi,e_i\}$ for each $i$. \end{lem}

\begin{proof} As noted before, $V$ is a finite $\ell^{\infty}$-sum
of mutually orthogonal Cartan factors and we need only consider convergence in each summand.
If $V$ does not have a spin factor summand, this has already been
proven in Lemma \ref{fr}. If $V$ contains a spin factor summand
with inner product $\langle\cdot,\cdot\rangle$, in which a spectral
decomposition of an element $z$ has the form
\[
z=\alpha_{1}d+\beta e=\alpha_{1}d+\beta\langle d,e^{*}\rangle d^{*}\qquad(0\leq\beta\leq\alpha_{1})
\]
(cf.\,\cite[(2.4)]{lie}) where the involution $*$ preserves weak convergence,
then one can use Lemma \ref{d} to conclude the proof. \end{proof}

\begin{lem}\label{eu*} Let $(e_{k})$ and $(u_{k})$ be two weakly
convergent sequences of minimal tripotents in $L(\mathbb{C}^{r},H)$,
with limits $e$ and $u$ respectively. If $e_{k}u_{k}^{\ast}=0$
from some $k$ onwards, then $eu^{\ast}=0$. \end{lem}

\begin{proof}

We have
\[
e_{k}=a_{k}\otimes b_{k}\quad{\rm {and}}\quad u_{k}=c_{k}\otimes d_{k}
\]
for some $a_{k},c_{k}\in\mathbb{C}^{r}$ and $b_{k},d_{k}\in H$,
of unit norm.

Pick two subsequences $(a_{j})$ and $(b_{j})$ of $(a_{k})$ and
$(b_{k})$ respectively such that $a\in\mathbb{C}^{r}$ is the norm
limit of $(a_{j})$ and $b\in H$ is the weak limit of $(b_{j})$.
Then we have the weak convergence $e=\text{weak-}\lim_{j}e_{j}=\text{weak-}\lim_{j}a_{j}\otimes b_{j}=a\otimes b$.
Likewise, $u=c\otimes d$ where $c$ is the norm limit of a subsequence $(c_{j'})$ of $(c_{j})$,
and $d$ the weak limit of a subsequence $(d_{j'})$ of $(d_{j})$.

Let $\mu\in\mathbb{C}^{r}$ and $h\in H$. Then for each $z\in L(\mathbb{C}^{r},H)$,
we have
\begin{eqnarray*}
\langle e_{}u^{\ast}z(\mu),h\rangle & = & \langle z(\mu),d\rangle\langle c,a\rangle\langle b,h\rangle\\
 & = & \lim_{j'}\langle z(\mu),d_{j'}\rangle\langle c_{j'},a_{j'}\rangle\langle b_{j'},h\rangle\\
 & = & \lim_{j'}\langle e_{j'}u_{j'}^{\ast}z(\mu),h\rangle
 =  0
\end{eqnarray*}
which implies $eu^*=0$.
\end{proof}

\begin{cor}\label{eu*e} Let $V$ be a finite-rank JB{*}-triple without
a spin factor direct summand. Let $(e_{k})$ and $(u_{k})$ be weakly
convergent sequences of minimal tripotents with limits $e$ and $u$
respectively such that $e_{k}$ and $u_{k}$ are orthogonal
for each $k$. Then $\{e,u,e\}=0$.

\end{cor}

\begin{proof}
We only need to verify the case where $V$ is the Cartan factor $L(\mathbb{C}^{r},H)$
for some Hilbert space $H$. We retain the notation from the previous
proof. As $(e_{k})$ and $(u_{k})$ are triple orthogonal, we have
$\langle c_{k},a_{k}\rangle=0$ by \textcolor{black}{Lemma} \ref{lem:Rank-one operators triple orth}
and therefore $e_{k}u_{k}^{\ast}=\langle\cdot,d_{k}\rangle\langle c_{k},a_{k}\rangle b_{k}=0$.
Hence Lemma \ref{eu*} finishes the proof.
\end{proof}

\begin{lem}\label{perp} Let $V$ be a JB{*}-triple of finite rank
$p$ and $(z_{k})$ a sequence in $D$ norm converging to some $\xi\in\overline{D}$,
with spectral decomposition
\[
z_{k}=\alpha_{1k}e_{1k}+\cdots+\alpha_{pk}e_{pk}\qquad(\alpha_{1k},\ldots,\alpha_{pk}\geq0)
\]
where each sequence $(e_{jk})_{k}$ weakly converges to $e_{j}$ and $(\alpha_{jk})_k$ converges
to $\alpha_j$, for
$j=1,\ldots,p$. Then we have $\alpha_{j}\{e_{i},e_{j},e_{i}\}=0$
for $i,j\in\{1,\ldots,p\}$ and $i\neq j$.\end{lem}

\begin{proof} If $e_{i}$ and $e_{j}$ do not belong to the same
direct summand then $\{e_{i},e_{j},e_{i}\}=0$. If they do, then the
result follows from Lemma \ref{d} and Corollary \ref{eu*e}, by considering
each summand in the decomposition of $V$ into Cartan factors.
\end{proof}

\begin{lem}\label{min}

Let $V$ be a finite-rank JB{*}-triple without a spin factor direct
summand. Then a weakly convergent sequence $(e_{k})$ of minimal tripotents
in $V$ with a minimal tripotent limit is norm convergent.

\end{lem}

\begin{proof} We only need to consider the case where $V$ is a finite
$\ell^{\infty}$-sum $V_{1}\oplus\cdots\oplus V_{q}$ of Type I Cartan
factors. Let $c=(c_{1},\ldots,c_{q})\in V$ be the weak limit of the
sequence $(e_{k})$ and write $e_{k}=(e_{1k},\ldots,e_{qk})$ which
has only one nonzero coordinate. Since $c$ is a minimal tripotent,
there exists $j\in\{1,\ldots,q\}$ such that $c_{j}$ is a minimal
tripotent in $V_{j}$ and $c_{i}=0$ for all $i\neq j$. By coordinatewise
weak convergence, there exists $K$ such that $k\geq K$ implies $e_{ik}=0$
for $i\neq j$, and $e_{jk}$ is a minimal tripotent in $V_{j}$.

Let $V_{j}=L(\mathbb{C}^{r},H)$ for some Hilbert space $H$ in which
case, we have
\[
e_{jk}=a_{k}\otimes b_{k}\quad (a_{k}\in\mathbb{C}^{r},~b_{k}\in H)
\]
and $\|a_{k}\|=\|b_{k}\|=1$.

To complete the proof, we show that every subsequence of $(e_{jk})$
has a subsequence norm converging to $c_{j}$. Let $(a_{k'}\otimes b_{k'})$
be a subsequence of $(e_{jk})$. We can find a subsequence $(a_{k''}\otimes b_{k''})$
of $(a_{k'}\otimes b_{k'})$ which weakly converges to $a''\otimes b''$
with $a''=\lim_{k''}a_{k''}$ and $b''=\text{weak-}\lim_{k''}b_{k''}$.
It follows that $c_{j}=a''\otimes b''$, which implies
$\|b''\|=1$ and $(b_{k''})$ actually norm converges to $b''$ in the Hilbert space $H$. A
simple calculation then shows that $(a_{k''}\otimes b_{k''})$ norm
converges to $c_{j}$.
\end{proof}

Now let $(z_{k})$
be a sequence in $D$ norm converging to some $\xi\in\partial D$.
As in (\ref{sp1}) and (\ref{sp2}), choose a spectral decomposition
of each $z_{k}$:
\begin{eqnarray}\label{sp}
z_{k} & = & \alpha_{1k}e_{1k}+\dots+\alpha_{pk}e_{pk},
\end{eqnarray}
where $\alpha_{1k},\ldots,\alpha_{pk}\geq0$ and $\{e_{1k},\dots,e_{pk}\}$
is a family of mutually orthogonal minimal tripotents in $\partial D$.
By weak compactness of  $\overline{D}$,
each sequence $(e_{ik})_{k}$ has a weak limit point $e_{i}$ say, for $i= 1,\ldots, p$.
Replace $(z_k)$ by a subsequence if necessary, we may assume henceforth  that $(\alpha_{ik})$
converges to $\alpha_{i}$,  and that  $(e_{ik})$  weakly converges to $e_{i}$
for $i=1,\dots,p$. It follows that
\[
\xi=\alpha_{1}e_{1}+\cdots+\alpha_{p}e_{p}.
\]

\begin{lem} \label{free-conv} Let $V$ be a JB{*}-triple of finite rank
$p$ and $(z_{k})$ be the sequence norm converging to $\xi$ as defined
above. Then there exists a non-empty set $J\subset\{1,\ldots,p\}$
such that \global\long\def\labelenumi{(\roman{enumi})}

\begin{enumerate}
\item[\rm(i)] $\alpha_{i}\neq0$ and  $(e_{ik})_{k}$ norm converges to
$e_{i}$, for each $i\in J$;
\item[\rm(ii)]  $\{e_{i}:i\in J\}$ is a family of pairwise
orthogonal minimal tripotents;
\item[\rm (iii)] $\alpha_{i}=0$ and $(\alpha_{ik}e_{ik})$ norm converges to $0$,
for each $i\in\{1,\ldots,p\}\backslash J$.
\end{enumerate}
\end{lem}

\begin{proof} For $i\in\{1,\ldots,p\}$, we have
\begin{eqnarray*}
\{e_{ik},z_{k},e_{ik}\}=\sum_{j=1}^{p}\alpha_{jk}\{e_{ik},e_{jk},e_{ik}\}=\alpha_{ik}e_{ik}
\end{eqnarray*}
which converges weakly to both $\alpha_{i}e_{i}$ and, by Lemma \ref{fr'},
to
\begin{eqnarray*}
\{e_{i},\xi,e_{i}\} & = & \sum_{j=1}^{p}\alpha_{j}\{e_{i},e_{j},e_{i}\}=\alpha_{i}\{e_{i},e_{i},e_{i}\}+\sum_{j\neq i}\alpha_{j}\{e_{i},e_{j},e_{i}\}.
\end{eqnarray*}
This, together with Lemma \ref{perp}, gives
\[
\alpha_{i}e_{i}=\alpha_{i}\{e_{i},e_{i},e_{i}\}+\sum_{j\neq i}\alpha_{j}\{e_{i},e_{j},e_{i}\}=\alpha_{i}\{e_{i},e_{i},e_{i}\}.
\]
Therefore either $\alpha_{i}=0$ or $e_{i}$ is a tripotent. If $\alpha_{i}e_{i}\neq0$,
then Lemma \ref{d} and Lemma \ref{fr} imply that $e_{i}$ is a minimal
tripotent and also, it follows from Lemma \ref{d} and Lemma \ref{min}
that the sequence $(e_{ik})$ actually norm converges to $e_{i}$.
Let
\begin{eqnarray*}
J=\{i\in\{1,\dots,p\}\medspace:\medspace\alpha_{i}e_{i}\neq0\}.
\end{eqnarray*}
Then $J\neq\emptyset$ as $\sum_{i=1}^{p}\alpha_{i}e_{i}=\xi\neq0$.

For each $i\in J$, we have $\alpha_{i}\neq 0$ and by norm convergence of $(e_{ik})$, the
minimal tripotents $\{e_i : i\in J\}$ are pairwise orthogonal.
This proves (i) and (ii).

To show (iii), let $j\in\{1,\ldots,p\}\backslash J$. Then we have
the norm convergence
\begin{eqnarray*}
\alpha_{jk}=\|\alpha_{jk}e_{jk}\| & \leq & \max\{\|\alpha_{ik}e_{ik}\|: i\in\{1,\ldots,p\}\backslash J\} = \|\sum_{i\in\{1,\ldots,p\}\backslash J}\alpha_{ik}e_{ik}\|\\
 & = & \|z_{k}-\sum_{i\in J}\alpha_{ik}e_{ik}\|~\longrightarrow~\|\xi-\sum_{i\in J}\alpha_{i}e_{i}\|=0\\
\end{eqnarray*}
which gives $\alpha_{j}=0$. The second assertion in (iii) follows from $\|e_{ik}\|=1$. \end{proof}

\begin{rem} \label{rem:Convergence of Peirce projections-1} Let $V$
be a finite-rank  JB{*}-triple and let $(z_{k})$, $\xi$ and
$J$ be as defined in Lemma \ref{free-conv}. The norm convergence
of the sequence $(e_{ik})_{k}$ to $e_{i}$ for all $i\in J$ and the
pairwise orthogonality of the minimal tripotents $\{e_{i}:i\in J\}$  ensure the norm
convergence
$$ \lim_{k\rightarrow \infty} P_{ij}(e_{sk}\medspace:\medspace s\in J)=
P_{ij}(e_{s}\medspace:\medspace s\in J)$$
of a sequence of joint Peirce projections, for
$i,j\in\{0\}\cup J$. Moreover, if $(w_{k})$ is a sequence in $\overline{D}$
weakly converging to $w\in\overline{D}$,
then we also have $P_{ij}(e_{sk}\medspace:\medspace s\in J)(w_{k})\wto P_{ij}(e_{s}\medspace:\medspace s\in J)(w)$
for $i,j\in\{0\}\cup J$.\end{rem}

Let $P_{jj'}^{k}$ denote the joint Peirce projections $P_{jj'}(e_{1k},\dots,e_{pk})$
for $0\leq j\leq j'\leq p$. By (\ref{pijek}) and \eqref{eq:Bergmann Neg Sq Rt},
we have
\begin{eqnarray*}
(1-r_{k}^{2})B(r_{k}z_{k},r_{k}z_{k})^{-1/2}(z_{k}) & = & (1-r_{k}^{2})\sum_{0\leq j\leq j'\leq p}(1-r_{k}^{2}\alpha_{jk}^{2})^{-1/2}(1-r_{k}^{2}\alpha_{j'k}^{2})^{-1/2}P_{jj'}^{k}(z_{k})\\
 & = & \sum_{j=1}^{p}\frac{1-r_{k}^{2}}{1-r_{k}^{2}\alpha_{jk}^{2}}\alpha_{jk}e_{jk}
\end{eqnarray*}
where $\alpha_{0k}$ is defined to be $0$ for all $k$. We can write
\begin{eqnarray}
\frac{1-r_{k}^{2}}{1-r_{k}^{2}\alpha_{jk}^{2}}  \, = \, \frac{\frac{1-r_{k}^{2}}{1-\|z_{k}\|^{2}}}
{\frac{1-r_{k}^{2}}{1-\|z_{k}\|^{2}}+\left(\frac{1-\alpha_{jk}{}^{2}}{1-\|z_{k}\|^{2}}\right)r_{k}^{2}}
 \, = \, \frac{\left(\frac{1-\|z_{k}\|^{2}}{1-\alpha_{jk}{}^{2}}\right)\lambda}
  {\left(\frac{1-\|z_{k}\|^{2}}{1-\alpha_{jk}{}^{2}}\right)\lambda+r_{k}^{2}}\label{eq:coeff limit sum-1}
\end{eqnarray}
where $\lambda(1-\|z_k\|^2) = 1-r_k^2$ and $\frac{1-\|z_{k}\|^{2}}{1-\alpha_{jk}{}^{2}}\in(0,1]$. Let
\begin{equation}
\sigma_{j}:=\limsup_{k}\frac{1-\|z_{k}\|^{2}}{1-\alpha_{jk}{}^{2}}\in[0,1]
\qquad (j=1, \ldots,p).\label{eq:sigma def-1}
\end{equation}
The sequence in (\ref{eq:coeff limit sum-1}) admits a convergent subsequence
\begin{equation}
\frac{1-r_{k'}^{2}}{1-r_{k'}^{2}\alpha_{jk'}^{2}} \longrightarrow \, \frac{\sigma_{j}\lambda}{\sigma_{j}\lambda+1} \quad {\rm as} \quad k' \rightarrow \infty.
\end{equation}
Noting that $\alpha_j \in [0,1]$ where  $\alpha_{j}<1$ implies $\sigma_{j}=0$, and $\alpha_{j}=0$
for $j\in\{1,\ldots,p\}\backslash J$, we have the following norm
convergence:
\begin{equation}
(1-r_{k'}^{2})B(r_{k'}z_{k'},r_{k'}z_{k'})^{-1/2}(z_{k'})\longrightarrow
\sum_{j=1}^{p}\frac{\sigma_{j}\lambda}{\sigma_{j}\lambda+1}\alpha_je_{j}=\sum_{j\in J}\frac{\sigma_{j}\lambda}{\sigma_{j}\lambda+1}e_{j}.\label{eq:c_k(lambda) limit-1}
\end{equation}

\begin{rem} \label{rem: norm frequently-1}Since $\|z_{k}\|=\max \{\alpha_{jk}: j=1, \ldots, p\}$,
there are two possibilities for each $j\in\{1,\dots,p\}$, namely,
either $\alpha_{jk}<\|z_{k}\|$ from some $k$ onwards or there is a subsequence $(\alpha_{jk'})_{k'}$
of $(\alpha_{jk})_{k}$ with $\alpha_{jk'}=\|z_{k'}\|$. As $\{1,\dots,p\}$
is finite, there exists some $j_{0}$ which satisfies the latter
in which case $\sigma_{j_{0}}=1$ and $\alpha_{j_0} =1$.\end{rem}

We are now ready to give an explicit description of the closed horoball $S(\xi,\lambda)$ in terms
of a Bergmann operator.

\begin{thm} \label{thm:explicit description of horosphere} Let $D$ be a
bounded symmetric domain of finite rank $p$ and
$f:D\to D$ a fixed-point free compact holomorphic map. Then there
exist a non-empty set $J\subset\{1,\ldots,p\}$ and a sequence $(z_{k})$
in $D$ converging to a boundary point
\[
\xi=\sum_{j\in J}\alpha_{j}e_{j}\qquad(0 <\alpha_{j} \leq 1)
\]
where $\{e_{j}:j\in J\}$ consists of mutually orthogonal minimal
tripotents in $\partial D$, such that for each $\lambda>0$,
\begin{eqnarray*}
S(\xi,\lambda) & = & \sum_{j\in J}\frac{\sigma_{j}\lambda}{1+\sigma_{j}\lambda}e_{j}+B\left(\sum_{j\in J}\sqrt{\frac{\sigma_{j}\lambda}{1+\sigma_{j}\lambda}}\,e_{j}\,,~\sum_{j\in J}\sqrt{\frac{\sigma_{j}\lambda}{1+\sigma_{j}\lambda}}\,e_{j}\right)^{1/2}(\overline{D})
\end{eqnarray*}
with $\sigma_{j}\geq 0$ and $\max\{\sigma_{j}:j\in J\}=1$.\end{thm}

\begin{proof} Let $D$ be realised as
 the open unit ball of a JB{*}-triple $V$ of rank $p$.
Let $(z_{k})$ be the sequence used in the construction of $S(\xi, \lambda)$ in (\ref{hball}),
where $\lim_k z_k =\xi$ and as in (\ref{sp}), we fix a spectral decomposition
$$z_{k} =  \alpha_{1k}e_{1k}+\dots+\alpha_{pk}e_{pk}$$
with $\alpha_j = \lim_k \alpha_{jk}$ for $j =1, \ldots, p$.
Throughout the proof
$\sigma_{j}$ is defined as in \eqref{eq:sigma def-1} and we take
Remark \ref{rem: norm frequently-1} into account.

By Lemma \ref{free-conv}, there exists a nonempty set $J\subset\{1,\dots,p\}$
such that $(e_{jk})_k$ norm converges to a minimal tripotent $e_{j}$ for all $j\in J$,
the tripotents  $\{e_{s}: s\in J\}$ are mutually
orthogonal and $\alpha_{j}=0$ for $j\in\{1,\dots,p\}\backslash J$. We denote the latter set
by $J^{c}$ to simplify notation.

Let $x\in S(\xi,\lambda)$. Then we have $x= \lim_k x_k$, where $x_k$ is an element in the Kobayashi ball $D_{k}(\lambda)$ defined in (\ref{kball}). By (\ref{kball1}), each $x_k$ has the form
$$
x_{k} =  c_{k} + r_{k}B(z_{k},z_{k})^{1/2}B(r_{k}z_{k},r_{k}z_{k})^{-1/2}(w_{k})
$$
 where
$ c_{k} = (1-r_{k}^{2})B(r_{k}z_{k},r_{k}z_{k})^{-1/2}(z_{k})$
and  $w_{k}\in{D}$.

To compute the norm limit $\lim_k x_k$, it suffices to compute a weak subsequential limit
$\lim_{k'} x_{k'}$.
 By weak compactness and by (\ref{eq:c_k(lambda) limit-1}), we may assume, by choosing
subsequences if necessary, that $(w_{k})$  weakly converges to some
$w\in\overline{D}$ and that
\begin{equation}\label{c}
\sigma_j = \lim_k\frac{1-\|z_{k}\|^{2}}{1-\alpha_{jk}^{2}}, \quad
\lim_k c_k =
\sum_{j\in J}\frac{\sigma_{j}\lambda}{1+\sigma_{j}\lambda}e_{j}.
\end{equation}
By the formulae for the square roots of the Bergmann operator in \eqref{eq:Bergmann Sq Rt}
and \eqref{eq:Bergmann Neg Sq Rt}, we have
\begin{eqnarray}
 &  & r_{k}B(z_{k},z_{k})^{1/2}B(r_{k}z_{k},r_{k}z_{k})^{-1/2}(w_k)\nonumber \\
 & = & r_{k}B(z_{k},z_{k})^{1/2}\sum_{0\leq j\leq j'\leq p}(1-r_{k}^{2}\alpha_{jk}^{2})^{-1/2}(1-r_{k}^{2}\alpha_{j'k}^{2})^{-1/2}\,P_{jj'}(e_{1k},\dots,e_{pk})(w_k)\nonumber \\
 & = & r_{k}\sum_{0\leq j\leq j'\leq p}\sqrt{\frac{1-\alpha_{jk}^{2}}{1-r_{k}^{2}\alpha_{jk}^{2}}}
 \sqrt{\frac{1-\alpha_{j'k}^{2}}{1-r_{k}^{2}\alpha_{j'k}^{2}}}\,P_{jj'}(e_{1k},\dots,e_{pk})(w_k)
 \label{eq:Horosphere noncentre-2}
\end{eqnarray}
where $\alpha_{0k}=0$ for all $k$. Note that
 $\alpha_{j}=0$ implies
 $\lim_k \frac{1-\alpha_{jk}^{2}}{1-r_{k}^{2}\alpha_{jk}^{2}} =1$
 and also, $\sigma_{j}=0$.

To compute the sum in (\ref{eq:Horosphere noncentre-2}), we split it into $3$ summands, over the following $3$ sets
of indices:
\begin{eqnarray*}
I &=& \{0\leq j\leq j'\leq p : j, j' \in \{0\} \cup J^c\},\\
II &=& \{(j, j'): j \in \{0\} \cup J^c, j' \in J\},\\
III &=& \{0\leq j\leq j'\leq p : j, j' \in  J\}.
\end{eqnarray*}
We can write the first summand as follows.
\begin{eqnarray*}
\sum_I &=& \sum_{\tiny\begin{array}{c}
j\leq j'\\
j,j'\in\{0\}\cup J^{c}
\end{array}}P_{jj'}(e_{1k},\dots,e_{pk})(w_{k})\\
&& -\sum_{\tiny\begin{array}{c}
j\leq j'\\
j,j'\in\{0\}\cup J^{c}
\end{array}}\left(1-\sqrt{\frac{1-\alpha_{jk}^{2}}{1-r_{k}^{2}\alpha_{jk}^{2}}}
\sqrt{\frac{1-\alpha_{j'k}^{2}}{1-r_{k}^{2}\alpha_{j'k}^{2}}}\,\right)P_{jj'}(e_{1k},\dots,e_{pk})(w_{k}).
\end{eqnarray*}
By Lemma \ref{cor: Peirce alg}\global\long\def\labelenumi{(\roman{enumi}}(iii), we have
\begin{eqnarray*}
\sum_I &=& P_{00}(e_{sk}\medspace:\medspace s\in J)(w_{k})\\
&& -\sum_{\tiny\begin{array}{c}
j\leq j'\\
j,j'\in\{0\}\cup J^{c}
\end{array}}\left(1-\sqrt{\frac{1-\alpha_{jk}^{2}}{1-r_{k}^{2}\alpha_{jk}^{2}}}
\sqrt{\frac{1-\alpha_{j'k}^{2}}{1-r_{k}^{2}\alpha_{j'k}^{2}}}\,\right)P_{jj'}(e_{1k},\dots,e_{pk})(w_{k}).
\end{eqnarray*}
For the second summand, Lemma \ref{cor: Peirce alg}\global\long\def\labelenumi{(\roman{enumi}}(ii) enables us
to write
\begin{eqnarray*}
\sum_{II} &=& \sum_{j'\in J}\sqrt{\frac{1}{r_{k}^{2}}\left(1-\frac{1-r_{k}^{2}}{1-r_{k}^{2}\alpha_{j'k}^{2}}\right)}\left(\sum_{j\in\{0\}\cup J^{c}}P_{jj'}(e_{1k},\dots,e_{pk})(w_{k})\right.\\& &~~ -\left.  \sum_{j\in\{0\}\cup J^{c}}\left(1-\sqrt{\frac{1-\alpha_{jk}^{2}}{1-r_{k}^{2}\alpha_{jk}^{2}}}\,\right)P_{jj'}(e_{1k},\dots,e_{pk})(w_{k})
\right)\\
&=& \sum_{j'\in J}\sqrt{\frac{1}{r_{k}^{2}}\left(1-\frac{1-r_{k}^{2}}{1-r_{k}^{2}\alpha_{j'k}^{2}}\right)}
\mbox{\Huge$($}P_{0j'}(e_{sk}\medspace:\medspace s\in J)(w_{k}) \\ && ~~ -\sum_{j\in\{0\}\cup J^{c}}\left(1-\sqrt{\frac{1-\alpha_{jk}^{2}}{1-r_{k}^{2}\alpha_{jk}^{2}}}\,\right)P_{jj'}(e_{1k},\dots,e_{pk})(w_{k})
\mbox{\Huge$)$}.
 \end{eqnarray*}
By Lemma \ref{cor: Peirce alg}\global\long\def\labelenumi{(\roman{enumi}}(i), the third summand can be written as,
$$ \sum_{III} = \sum_{\tiny\begin{array}{c}
j\leq j'\\
j,j'\in J
\end{array}}\sqrt{\frac{1}{r_{k}^{2}}\left(1-\frac{1-r_{k}^{2}}{1-r_{k}^{2}\alpha_{jk}^{2}}\right)}
\sqrt{\frac{1}{r_{k}^{2}}\left(1-\frac{1-r_{k}^{2}}{1-r_{k}^{2}\alpha_{j'k}^{2}}\right)}\,
P_{jj'}(e_{sk}\medspace:\medspace s\in J)(w_{k}).$$

Noting that all Peirce projections are contractive and by Remark \ref{rem:Convergence of Peirce projections-1}
as well as \eqref{eq:Bergmann Sq Rt}, we have the weak convergence
\begin{eqnarray*}
&&\mbox{weak-}\lim_k r_{k}B(z_{k},z_{k})^{1/2}B(r_{k}z_{k},r_{k}z_{k})^{-1/2}(w_k)
= \mbox{weak-}\lim_k r_k \,\left(\sum_I \,+\, \sum_{II} \,+\, \sum_{III}\right)\\
&=&
P_{00}(e_{s}\medspace:\medspace s\in J)(w)
  +\sum_{j'\in J}\sqrt{1-\frac{\sigma_{j'}\lambda}{1+\sigma_{j'}\lambda}}\,P_{0j'}(e_{s}\medspace:\medspace s\in J)(w)\\
&&  +\sum_{\tiny\begin{array}{c}
j\leq j'\\
j,j'\in J
\end{array}}\sqrt{1-\frac{\sigma_{j}\lambda}{1+\sigma_{j}\lambda}}\sqrt{1-\frac{\sigma_{j'}\lambda}
{1+\sigma_{j'}\lambda}}\,P_{jj'}(e_{s}\medspace:\medspace s\in J)(w)\\
 & = & B\left(\sum_{j\in J}\sqrt{\frac{\sigma_{j}\lambda}{1+\sigma_{j}\lambda}}\,e_{j},\,\sum_{j\in J}\sqrt{\frac{\sigma_{j}\lambda}{1+\sigma_{j}\lambda}}\,e_{j}\right)^{1/2}(w).
\end{eqnarray*}

Now, together with (\ref{c}), we have
\[
x=\lim_{k}x_{k}=\sum_{j\in J}\frac{\sigma_{j}\lambda}{1+\sigma_{j}\lambda}e+B\left(\sum_{j\in J}\sqrt{\frac{\sigma_{j}\lambda}{1+\sigma_{j}\lambda}}\,e_{j},\, \sum_{j\in J}\sqrt{\frac{\sigma_{j}\lambda}{1+\sigma_{j}\lambda}}\,e_{j}\right)^{1/2}(w)
\]
which belongs to $\sum_{j\in J}\frac{\sigma_{j}\lambda}{1+\sigma_{j}\lambda}e_{j}+B\left(\sum_{j\in J}\sqrt{\frac{\sigma_{j}\lambda}{1+\sigma_{j}\lambda}}\,e_{j}\,,~\sum_{j\in J}\sqrt{\frac{\sigma_{j}\lambda}{1+\sigma_{j}\lambda}}\,e_{j}\right)^{1/2}(\overline{D})$.

Conversely, let
\begin{align*}
y & \in\sum_{j\in J}\frac{\sigma_{j}\lambda}{1+\sigma_{j}\lambda}e_{j}+B\left(\sum_{j\in J}\sqrt{\frac{\sigma_{j}\lambda}{1+\sigma_{j}\lambda}}\,e_{j},\, \sum_{j\in J}\sqrt{\frac{\sigma_{j}\lambda}{1+\sigma_{j}\lambda}}\,e_{j}\right)^{1/2}(\overline{D}).
\end{align*}
 We show $y\in S(\xi,\lambda)$. There exists some $x\in\overline{D}$
such that
\begin{align*}
y & =\sum_{j\in J}\frac{\sigma_{j}\lambda}{1+\sigma_{j}\lambda}e_{j}+B\left(\sum_{j\in J}\sqrt{\frac{\sigma_{j}\lambda}{1+\sigma_{j}\lambda}}\,e_{j},\, \sum_{j\in J}\sqrt{\frac{\sigma_{j}\lambda}{1+\sigma_{j}\lambda}}\,e_{j}\right)^{1/2}(x).
\end{align*}
Let $$y_{k}= (1-r_{k}^{2})B(r_{k}z_{k},r_{k}z_{k})^{-1/2}(z_{k})  +r_{k}B(z_{k},z_{k})^{1/2}B(r_{k}z_{k},r_{k}z_{k})^{-1/2}(x).$$
Then $y_{k}\in D_{k}(\lambda)$. Repeating the convergence arguments as before,
with $x$ in  place of $w_k$, but with
 norm convergence as opposed to weak convergence, one sees that
 $ y= \lim_k y_k \in S(\xi, \lambda)$, which completes the proof.
\end{proof}

\begin{rem}\label{reform}
By re-ordering the index set $J$ in Theorem \ref{thm:explicit description of horosphere}, the description of
the boundary point $\xi$ and $S(\xi,\lambda)$ can be reformulated more simply as:
$$\xi = \sum_{j=1}^m \alpha_j e_j \qquad (\alpha_j >0)$$
for some $m \in \{1, \ldots, p\}$ and
$$S(\xi, \lambda)  =  \sum_{j=1}^m\frac{\sigma_{j}\lambda}{1+\sigma_{j}\lambda}\,e_{j}+B\left(\sum_{j=1 }^m\sqrt{\frac{\sigma_{j}\lambda}{1+\sigma_{j}\lambda}}\,e_{j}, ~\sum_{j=1 }^m\sqrt{\frac{\sigma_{j}\lambda}{1+\sigma_{j}\lambda}}\,e_{j}\right)^{1/2}(\overline{D})
$$ where $\sigma_j \geq 0$ and $\max\{\sigma_j: j=1, \ldots,m\}=1$.
\end{rem}

For finite rank bounded symmetric domains, we now have the following generalisation of
Wolff's theorem.

\begin{thm}\label{interior} Let $f$  be a fixed-point free compact holomorphic self-map on a  bounded symmetric domain $D$ of finite rank  $p$.
Then there is a sequence $(z_k)$ in $D$ converging
to a boundary point $$ \xi=\sum_{j=1}^m\alpha_{j}e_{j} \qquad (\alpha_{j} >0,\, m \in \{1, \ldots,p\})$$
 where $e_1, \ldots , e_m$ are  orthogonal
minimal tripotents in $\partial D$, such that for each $\lambda >0$, the convex $f$-invariant domain $H(\xi, \lambda)$ is the horoball $S_0(\xi,\lambda)$ at $\xi$, which has the form
\begin{eqnarray}\label{511}
S_0(\xi, \lambda) & = & \sum_{j=1}^m\frac{\sigma_{j}\lambda}{1+\sigma_{j}\lambda}\,e_{j}+B\left(\sum_{j=1 }^m\sqrt{\frac{\sigma_{j}\lambda}{1+\sigma_{j}\lambda}}\,e_{j}, ~\sum_{j= 1}^m\sqrt{\frac{\sigma_{j}\lambda}{1+\sigma_{j}\lambda}}\,e_{j}\right)^{1/2}({D})
\end{eqnarray}
and is affinely homeomorphic to $D$, where $\sigma_{j}\geq 0$ and $\max\{\sigma_{j}:j=1, \ldots, m\}=1$.
\end{thm}

\begin{proof} As before, we identify $D$ as the open unit ball of a JB*-triple $V$ of rank $p$.
Let $(z_k)$ be the sequence in  Theorem \ref{thm:explicit description of horosphere} with limit
$\xi = \sum_{j=1}^m \alpha_je_j$   as in Remark \ref{reform}.
 Let $\{\sigma_j: j=1, \ldots, m \}$ be the same as in Remark \ref{reform}.
Observe that
the map $\varphi : V \rightarrow V$ defined by
$$ \varphi = \sum_{j=1}^m\frac{\sigma_{j}\lambda}{1+\sigma_{j}\lambda}\,e_{j}+B\left(\sum_{j=1 }^m\sqrt{\frac{\sigma_{j}\lambda}{1+\sigma_{j}\lambda}}\,e_{j}, ~\sum_{j=1 }^m\sqrt{\frac{\sigma_{j}\lambda}{1+\sigma_{j}\lambda}}\,e_{j}\right)^{1/2}$$
is an affine homeomorphism since the Bergmann operator $B(a,a)^{1/2}$ is invertible
 whenever $\|a\| <1$. Hence $\varphi (D)$ is the interior of
$\varphi (\overline D) = S(\xi,\lambda)$, that is, $S_0(\xi,\lambda) = \varphi (D)$
which proves (\ref{511}). It also shows that $S_0(\xi,\lambda)$ and $D$ are affinely
homeomorphic. The equality  $H(\xi, \lambda) = S_0(\xi, \lambda)$ is shown in the next theorem.
\end{proof}

\begin{rem}\label{final} For $p=1$, that is, $D$ is a Hilbert ball, in the above theorem,
we have
$$S_0(\xi, \lambda) = \frac{\lambda}{1+\lambda} \xi + B\left(\sqrt{\frac{\lambda}{1+\lambda}}\xi\,,\sqrt{\frac{\lambda}{1+\lambda}}
\xi\right)^{1/2}({D})$$ and in one dimension, that is, $D=\mathbb{D}$, it reduces to Wolff's horodisc in (\ref{horo1}):
$$S_0(\xi, \lambda) = \frac{\lambda}{1+\lambda} \xi + \frac{1}{1+\lambda}\, \mathbb{D}.$$
\end{rem}

\begin{thm}\label{inv} Let $f$ be a fixed-point free compact holomorphic self-map on a finite-rank bounded symmetric domain $D$. For each $\lambda >0$, let $H(\xi, \lambda)$ be the $f$-invariant domain defined in Theorem \ref{g}. Then we have $H(\xi, \lambda) =S_0(\xi,\lambda)$ and the closure $\overline H (\xi, \lambda)$ satisfies
$$\overline H(\xi, \lambda)\cap D = S(\xi, \lambda)\cap D = \left\{x\in  D: \limsup_{k\rightarrow \infty} \frac{1-\|z_k\|^2}{1-\|g_{-z_k}(x)\|^2} \leq \frac{1}{\lambda}\right\}.$$

\end{thm}

\begin{proof}
By Lemma \ref{subset},
it suffices to show $S_0(\xi,\lambda) \subset H(\xi, \lambda)$ for the first assertion.
By (\ref{511}), we have
$$S_0(\xi,\lambda) = c(\lambda) + B(D)$$
where
$$c(\lambda) = \sum_{i=1}^m\frac{\sigma_{i}\lambda}{1+ \sigma_{i}\lambda}e_{i} \quad
{\rm and} \quad  B = B\left(\sum_{i=1}^m\sqrt{\frac{\sigma_{i}\lambda}{1 +\sigma_{i}\lambda}}e_{i}\,,~\sum_{i\in J}\sqrt{\frac{\sigma_{i}\lambda}{1 + \sigma_{i}\lambda}}e_{i}\right)^{1/2}.$$

Let $x\in S_0(\xi, \lambda)$. Then there exists $v \in D$ such that
\begin{eqnarray*}
x &=& c(\lambda) + B(v)\\
&=& \lim_k [(1-r_k^2)B(r_k,z_k,r_kz_k)^{-1/2}(z_k) + r_k B(z_k,z_k)^{1/2}B(r_kz_k,r_kz_k)^{-1/2}(v)].
\end{eqnarray*}
Let $x_k = (1-r_k^2)B(r_k z_k,r_kz_k)^{-1/2}(z_k) + r_k B(z_k,z_k)^{1/2}B(r_kz_k,r_kz_k)^{-1/2}(v).$
Then by (\ref{y}) we have $$v = g_{r_kz_k}\left(\frac{g_{-z_k}(x_k)}{r_k}\right)$$
which gives $g_{-z_k}(x_k) = r_k g_{-r_kz_k}(v)$ and
$$\limsup_{k\rightarrow \infty} \frac{1-\|z_k\|^2}{1-\|g_{-z_k}(x_k)\|^2} =\limsup_{k\rightarrow \infty} \frac{1-\|z_k\|^2}{1-r_k^2\|g_{-r_kz_k}(v)\|^2}$$
where, by choosing a subsequence, we may replace the {\it upper limit} $\limsup_k$ by the {\it limit} $\lim_k$ in the
following computation.
Since, by (\ref{gab}),
$$\|g_{-r_kz_k}(v)\|^2 = 1-  \frac{1}{\|B(v,v)^{-1/2}B(v, r_kz_k)B(r_kz_k, r_kz_k)^{-1/2}\|}\,,$$
we have
\begin{eqnarray}\label{lll}
\frac{1-r_k^2}{1-\|g_{-z_k}(x_k)\|^2} &=& \frac{1-r_k^2}{1-\|g_{-r_kz_k}(v)\|^2 + (1-r_k^2)\|g_{-r_kz_k}(v)\|^2}\nonumber\\
&=& \frac{(1-r_k^2)\|B(v,v)^{-1/2}B(v, r_kz_k)B(r_kz_k, r_kz_k)^{-1/2}\|}{r_k^2 +(1-r_k^2)\|B(v,v)^{-1/2}B(v, r_kz_k)B(r_kz_k, r_kz_k)^{-1/2}\|}.
\end{eqnarray}

We observe that the sequence $((1-r_k^2)\|B(v,v)^{-1/2}B(v, r_kz_k)B(r_kz_k, r_kz_k)^{-1/2}\|)$ is bounded by
$$\frac{(1-r_k^2)\|B(v,v)^{-1/2}\|(1+r_k \|v\|\|z_k\|)^2}{1-r_k^2\|z_k\|^2}\leq \|B(v,v)^{-1/2}\|(1+\|v\|)^2.$$
Passing to a subsequence, we may assume that the sequence converges to some limit $\ell \geq 0$.

 Hence we have, from (\ref{lll}),
$$ \lim_{k\rightarrow \infty} \frac{1-r_k^2}{1-\|g_{-z_k}(x_k)\|^2} = \frac{\ell}{1+\ell} <1.$$ It follows that
$$\lim_{k\rightarrow \infty}\frac{1-\|z_k\|^2}{1-\|g_{-z_k}(x)\|^2}= \lim_{k\rightarrow \infty}\frac{1-\|z_k\|^2}{1-\|g_{-z_k}(x_k)\|^2} = \lim_{k \rightarrow \infty}
\frac{1-\|z_k\|^2}{1-r_k^2}\frac{1-r_k^2}{1-\|g_{-z_k}(x_k)\|^2} = \left(\frac{1}{\lambda}\right)\left(
\frac{\ell}{1+\ell}\right) < \frac{1}{\lambda}$$
and $x \in H(\xi, \lambda)$. This proves $S_0(\xi,\lambda) \subset H(\xi, \lambda)$.

For the second assertion, we need only show the last equality. It has already been shown in
Lemma \ref{subset} that $S(\xi, \lambda)\cap D
 \subset \{x \in D: F(x) \leq 1/\lambda\}$. Conversely,
 let $x \in D$ and
$F(x) \leq 1/\lambda$. Then for $0<\vp<1$, we have $F(x) < \frac{1}{\vp \lambda}$ which implies
$x\in H(\xi,\vp \lambda) \subset S_0(\xi,\vp \lambda)$. Hence by (\ref{511}),
$$x  =  \sum_{j=1}^m\frac{\sigma_{j}\vp\lambda}{1+\sigma_{j}\vp\lambda}\,e_{j}+B\left(\sum_{j=1 }^m\sqrt{\frac{\sigma_{j}\vp\lambda}{1+\sigma_{j}\vp\lambda}}\,e_{j}, ~\sum_{j= 1}^m\sqrt{\frac{\sigma_{j}\vp\lambda}{1+\sigma_{j}\vp\lambda}}\,e_{j}\right)^{1/2}(x_\vp)$$
for some $x_\vp \in D$. Let $x_1\in \overline D$ be a weak limit point of $\{x_\vp: 0 < \vp <1\}$, and let $\vp \rightarrow 1$. Then we have
$$x  =  \sum_{j=1}^m\frac{\sigma_{j}\lambda}{1+\sigma_{j}\lambda}\,e_{j}+B\left(\sum_{j=1 }^m\sqrt{\frac{\sigma_{j}\lambda}{1+\sigma_{j}\lambda}}\,e_{j}, ~\sum_{j= 1}^m\sqrt{\frac{\sigma_{j}\lambda}{1+\sigma_{j}\lambda}}\,e_{j}\right)^{1/2}(x_1)\in S(\xi, \lambda).$$
This completes the proof.

\end{proof}

The previous two results show that $H(\xi, \lambda) \neq \emptyset$ for all $\lambda >0$ in finite rank bounded symmetric domains.
We conclude this section by showing that this is in fact the case for {\it all} bounded symmetric domains.
Let $D$ be a bounded symmetric domain realised as the open unit ball of a JB*-triple $V$. By the Gelfand-Naimark
theorem for JB*-triples \cite{fb}, $V$ can be realised as a closed subtriple of an $\ell^\infty$-sum $\bigoplus_\iota
V_\iota$ of Cartan factors $V_\iota$.

Given $r>0$, we show that $F^{-1}[0, r) \neq \emptyset$ for the function
 $F$ in Lemma \ref{f}, which would entail $H(\xi, 1/r) \neq \emptyset$.
 For this, we need to find an element $x\in D$ satisfying
$$\limsup_k \|B(x,x)^{-1/2}B(x, z_k) B(z_k,z_k)^{-1/2}\|(1-\|z_k\|^2) < r$$
where the sequence $(z_k)$ is as before, with limit $\xi \in \partial D$. We first observe that,  given
$a\in V$, the Bergmann operators $B(a,a)$ with respect to $V$ and with respect to $\bigoplus_\iota V_\iota$
coincide on $V$. For our purpose, we may therefore assume $V= \bigoplus_\iota V_\iota$ without loss of generality.
Further, since the triple product on  $\bigoplus_\iota V_\iota$ is defined coordinatewise, we have, for $a= \oplus_\iota a_\iota \in\bigoplus_\iota V_\iota$,
$$B(a,a) = \bigoplus_\iota B(a_\iota, a_\iota)$$
and it can be seen that, for $x= \oplus_\iota x_\iota$ and $z= \oplus_\iota z_\iota$ in V, we have
\begin{eqnarray}
&&\|B(x,x)^{-1/2}B(x, z) B(z,z)^{-1/2}\|(1-\|z\|^2)\nonumber\\ &=& \sup_\iota \|B(x_\iota,x_\iota)^{-1/2}B(x_\iota, z_\iota) B(z_\iota,z_\iota)^{-1/2}\|(1-\|z\|^2)\label{iota1}\\
&\leq&\sup_\iota \|B(x_\iota,x_\iota)^{-1/2}B(x_\iota, z_\iota) B(z_\iota,z_\iota)^{-1/2}\|(1-\|z_\iota\|^2)\label{iota2}.
\end{eqnarray}

\begin{cor}\label{h}
Let $f$ be a fixed-point free compact holomorphic self-map on a bounded symmetric domain $D$. Then
the $f$-invariant domain $H(\xi, \lambda)$ is non-empty for each $\lambda>0$.
\end{cor}
\begin{proof}
In view of previous remarks, we need only consider the case where $D$ is the open unit ball of
an $\ell^\infty$-sum $V= \bigoplus_\iota V_\iota$ of Cartan factors. As in Lemma \ref{f}, let
\[
F(x)=\limsup_{k\rightarrow\infty}\frac{1-\|z_{k}\|^{2}}{1-\|g_{-z_{k}}(x)\|^{2}}\qquad(x\in D).
\]
For each $r >0$, we show $F^{-1}[0, r) \neq \emptyset$  which would complete the proof.
Let
$\xi = \oplus_\iota \xi_\iota$ where $1= \|\xi\| = \sup_\iota \|\xi_\iota\|$. Since
\[
F(x) = \limsup_k \|B(x,x)^{-1/2}B(x, z_k) B(z_k,z_k)^{-1/2}\|(1-\|z_k\|^2),
\]
using (\ref{iota1}), it suffices to show that for each Cartan factor $V_\iota$,
there is some $x_\iota$ in its open unit ball $ D_\iota$ satisfying
\begin{equation}\label{r}
\limsup_k \|B(x_\iota,x_\iota)^{-1/2}B(x_\iota, z_{k,\iota}) B(z_{k,\iota},z_{k,\iota})^{-1/2}\|(1-\|z_k\|^2) < r
\end{equation}
where $z_k = \oplus_\iota z_{k,\iota}\in \bigoplus_\iota V_\iota$. If $\|\xi_\iota\| = \lim_k \|z_{k,\iota}\| <1$, then this is obvious since
$$\limsup_{k\rightarrow\infty}\frac{1-\|z_{k}\|^{2}}{1-\|g_{-z_{k,\iota}}(x_\iota)\|^{2}} = \frac{0}{1-\|g_{-\xi_\iota}(x_\iota)\|^{2}}=0$$
for any $x_\iota \in D_\iota$.
If $\|\xi_\iota\|=1$ and if $V_\iota$ is of finite rank, then Theorem \ref{interior} and (\ref{iota2}) also
implies such $x_\iota \in D_\iota$ exists.

 There remains the case of an {\it infinite-rank} Cartan factor $V_\iota$ with $\|\xi_\iota\|=1$. Such $V_\iota$  can be realised as a closed
subtriple of $\mathcal{L}(H,H)$ for some Hilbert space $H$. Similar remark about the Bergmann operator $B(a,a)$
as before allows us to assume $V_\iota = \mathcal{L}(H,H)$ without loss of generality. Let $t \in (0,1)$. For each $k$, define
a bounded linear operator $T_k : V_\iota \rightarrow V_\iota$ by
$$T_k(z) = (\mathbf{1}-tz_{k,\iota}z_{k,\iota}^*)z(\mathbf{1}-z_{k,\iota}^*z_{k,\iota}) \qquad (z \in V_\iota=\mathcal{L}(H,H)).$$
Then by (\ref{baa}), we have $T_k B(z_{k,\iota},z_{k,\iota})^{-1/2}(z) = (\mathbf{1}-tz_{k,\iota}z_{k,\iota}^*)(\mathbf{1}-z_{k,\iota}z_{k,\iota}^*)^{-1/2}z(\mathbf{1}-z_{k,\iota}^*z_{k,\iota})^{1/2}$.
Since
\begin{eqnarray}
&&\|(\mathbf{1}-z_{k,\iota}z_{k,\iota}^*)^{-1/2}z(\mathbf{1}-z_{k,\iota}^*z_{k,\iota})^{1/2} \|^2\nonumber\\ &=& \|(\mathbf{1}-z_{k,\iota}z_{k,\iota}^*)^{-1/2}z(\mathbf{1}-z_{k,\iota}^*z_{k,\iota})z^*
(\mathbf{1}-z_{k,\iota}z_{k,\iota}^*)^{-1/2}\|\nonumber\\&=&\|B(z_{k,\iota},z_{k,\iota})^{-1/2}(z(\mathbf{1}-z_{k,\iota}^*z_{k,\iota})z^*)\| \leq \frac{\|z\|^2}{1-\|z_{k,\iota}\|^2},
\label{onehalf}
\end{eqnarray}
we infer $\|T_k B(z_{k,\iota},z_{k,\iota})^{-1/2}\| \leq 1/\sqrt{1-\|z_{k,\iota}\|^2}$ and hence
\begin{equation}
\limsup_k \|B(x,x)^{-1/2}T_k B(z_{k,\iota},z_{k,\iota})^{-1/2}\|(1-\|z_{k,\iota}\|^2) \leq \lim_k \frac{\sqrt{1-\|z_{k,\iota}\|^2}}{1-\|x\|^2} =0
\label{zero}
\end{equation}
for  all $x \in D_\iota \subset V_\iota$.

Let $x_t = t\xi_\iota \in D_\iota$. Then we have
$$\lim_k \|B(x_t,x_t)^{-1/2}B(x_t,z_{k,\iota}) - B(x_t,x_t)^{-1/2}T_k\|
= \|B(x_t,x_t)^{-1/2}B(x_t,\xi_\iota) - B(x_t,x_t)^{-1/2}T\|$$
where $T:V_\iota \rightarrow V_\iota$ is the operator $T(z) = (\mathbf{1}-t\xi_\iota\xi_\iota^*)z(\mathbf{1}
-\xi_\iota^*\xi_\iota) $ and for $z \in V_\iota$,
\begin{eqnarray*}
 &&\|(B(x_t,x_t)^{-1/2}B(x_t,\xi_{\iota}) - B(x_t,x_t)^{-1/2}T)(z)\|\\
 &=& \|(\mathbf{1}-t^2\xi_\iota\xi_\iota^*)^{-1/2}(\mathbf{1}-t\xi_\iota\xi_\iota^*)z(\mathbf{1}- t\xi_{\iota}^*\xi_{\iota} -\mathbf{1}+\xi_{\iota}^*\xi_{\iota})(\mathbf{1}-t^2\xi_{\iota}^*\xi_{\iota})
 ^{-1/2}\|\\
 &=&\|(\mathbf{1}-t^2\xi_\iota\xi_\iota^*)^{-1/2}(\mathbf{1}-t^2\xi_\iota\xi_\iota^*+ (t^2-t)\xi_\iota\xi_\iota^*)z(1-t)\xi_{\iota}^*\xi_{\iota}(\mathbf{1}-t^2\xi_{\iota}^*\xi_{\iota})
 ^{-1/2}\|
 \\
&\leq& \|(\mathbf{1}-t^2\xi_\iota\xi_\iota^*)^{1/2}z(1-t)\xi_{\iota}^*\xi_{\iota}(\mathbf{1}-t^2\xi_{\iota}^*\xi_{\iota})
 ^{-1/2}\| \\&&+ \|(\mathbf{1}-t^2\xi_\iota\xi_\iota^*)^{-1/2} (t^2-t)\xi_\iota\xi_\iota^*z(1-t)\xi_{\iota}^*\xi_{\iota}(\mathbf{1}-t^2\xi_{\iota}^*\xi_{\iota})
 ^{-1/2}\|\\
&\leq& \frac{(1-t)\|z\|}{\sqrt{1-t^2\|\xi_\iota\|^2}} + \frac{(t-t^2)(1-t)\|z\|}{1-t^2\|\xi_\iota\|^2}
 \end{eqnarray*}
 where the last inequality follows from a computation similar to (\ref{onehalf}).
 This gives
 \begin{eqnarray*}
 &&\lim_k\|B(x_t,x_t)^{-1/2}B(x_t,z_{k,\iota}) - B(x_t,x_t)^{-1/2}T_k\|\\ &\leq& \frac{1-t}{\sqrt{1-t^2}} + \frac{(t-t^2)(1-t)}{1-t^2} = \sqrt{\frac{1-t}{1+t}} + \frac{t(1-t)}{1+t}.
 \end{eqnarray*}
It follows from this inequality and (\ref{zero}) that
$$\limsup_k\|B(x_t,x_t)^{-1/2}B(x_t,z_{k,\iota})B(z_{k,\iota},z_{k,\iota})^{-1/2}\|(1-\|z_k\|^2) \leq \sqrt{\frac{1-t}{1+t}} + \frac{t(1-t)}{1+t}$$ which can be made less than $r$ by choosing $t$ very close to $1$.
For such $t$, the element $x_t\in D_\iota$ satisfies (\ref{r}) and the proof is complete.
\end{proof}

\section{Iteration of holomorphic maps}

The Denjoy-Wolff theorem holds for fixed-point free compact
holomorphic self-maps on bounded symmetric domains of rank one,
which asserts convergence of the iterates to a single boundary
point \cite{cm1,h,m}. The rank-one domains are the Hilbert balls
of which the boundary points are exactly the {\it boundary components}
(defined below)  of
the boundary. On the other hand, the Denjoy-Wolff theorem fails
for the bidiscs. Instead of converging to a single boundary point,
 Herv\'e \cite{h1}  has shown by intricate and lengthy arguments
 that the iterates accumulate in
the closure of a single boundary component. This
suggests that a natural generalisation of the Denjoy-Wolff theorem
to other domains should be that the limit set of  a fixed-point free compact
holomorphic self-map
$f$ is contained in the closure of a single boundary component, where the limit
set of $f$ consists of the images of all subsequential limits
$\lim_k f^{n_k}$ of the iterates $(f^n)$, where $f^n =\underbrace{f \circ \cdots \circ f}_{n\mbox{-}{\rm times}}$.

As in the case of the complex unit disc, the generalised Wolff theorem in
Theorem \ref{interior} enables us to show in this section that for finite-rank bounded symmetric
domains, all images of subsequential limits with weakly closed
range are indeed contained in the closure of a single boundary component.

The concept of a boundary component of a convex domain $U$ in a
Banach space $Z$ has been introduced and studied in  \cite{ks,lo}. A
subset $C\subset \overline U$ is called a  {\it
boundary component} of the closure $\overline U$ if the following conditions
are satisfied:
\begin{enumerate}
\item[(i)] $C \neq \emptyset$; \item[(ii)] for each holomorphic
map $f: \mathbb{D}\longrightarrow Z$ with $f(\mathbb{D})\subset
\overline U$, either $f(\mathbb{D})\subset C$ or $f(\mathbb{D})
\subset \overline U \backslash C$; \item[(iii)] $C$ is minimal
with respect to (i) and (ii).
\end{enumerate}
Two boundary components  are either equal or
disjoint. The interior $U$ is the unique open boundary component of $\overline U$, all others
are contained in the boundary $\partial U$ \cite{ks}.  By a slight abuse of language,
we also called the latter boundary components {\it of} $\partial U$. For
each $a \in \overline U$, we denote by $K_a$ the boundary
component containing $a$.

Given a holomorphic map $h : U \rightarrow \overline U$, the image $h(U)$ is entirely
contained in a single boundary component of $\overline U$ (cf.\,\cite[Lemma 3.3]{lie}).

For a bounded symmetric domain $D$ realised as
the open unit ball of a JB*-triple $V$, the boundary component $K_e$ of a tripotent
$e$ is given by $K_e = e + V_0(e)\cap D$, where $V_0(e)= P_0(e)(V)$ is the Peirce $0$-space.
Moreover, if $V$ is of finite rank, then each boundary component of $\partial D$ is of this form
(cf.\,\cite[Proposition 4.3]{ks} and \cite[Theorem 6.3]{lo}). Write $D_e$ for the open unit
ball $ V_0(e)\cap D$ in the JB*-triple $V_0(e)$. Since $P_0(e)$ is a contractive projection, we see that
$D_e = P_0(e)(D)$ and $\overline K_e = e +\overline{ D_e}
= e+ P_0(e)(\overline D)$. Also, the boundary $\partial K_e$ of $K_e$ equals $e+ \partial D_e$.

Each tripotent $c$ in $V_0(e)$ is orthogonal to $e$ and its Peirce $0$-space in $V_0(e)$ is the eigenspace
$$(V_0(e))_0(c) = \{v\in V_0(e): (c \bo c)(v)=0\}.$$

\begin{lem}\label{ooc}
In the above notation, we have $(V_0(e))_0(c)= V_0(e+c)$.
\end{lem}
\begin{proof}
Considering the joint Peirce decomposition induced by $\{e,c\}$ and from
(\ref{eq:Peirce-0 of sum}), we have
$$V_0(e+c)=V_{00}(e,c) = V_0(e) \cap V_0(c) = (V_0(e))_0(c).$$

\end{proof}

We see from Lemma \ref{ooc} that each boundary component of $\partial D_e$ is of the form
$c + (V_0(e))_0(c)\cap D_e= c+V_0(e+c)\cap D$ for some tripotent $c \in V_0(e)$. Hence each boundary
component of $\partial K_e$ is of the form $e+c + V_0(e+c)\cap D=K_{e+c}$, which is also a boundary
component of $ \overline D$.

Given a compact holomorphic self-map $f$ on a bounded symmetric domain $D$, we call a holomorphic map
$h: D \rightarrow \overline D$ a {\it limit function} of the iterates $(f^n)$ if there is a subsequence
$(f^{n_k})$ of $(f^n)$ converging to $h$ locally uniformly.

\begin{rem}\label{normal}\rm It has been shown in
\cite[Lemma 1]{cm1} that  every subsequence of
$(f^{n_k})$ has a subsequence converging locally uniformly to a holomorphic map $h : D \rightarrow \overline D$.
It follows that if $(f^n)$ has a unique limit function $h$, then $(f^n)$ converges locally uniformly to $h$.
\end{rem}

\begin{thm}\label{dw} Let $D$ be a bounded symmetric domain of finite rank $p$
and let $f:D \rightarrow D$ be a compact fixed-point free holomorphic map. Then there
is a boundary point $\xi \in \partial D$ of the form
$$\xi = \sum_{j=1}^m \alpha_je_j \quad (\alpha_j >0,\, m \leq p)$$ for some orthogonal
tripotents $e_1, \ldots, e_m \in \partial D$, such that
for each limit function $h$ of $(f^n)$ with weakly closed range, we have $h(D) \subset \overline K_e$,
where $K_e$ is the boundary component of $e = e_1 + \cdots +e_m$ in $\partial D$.
\end{thm}
\begin{proof} Let $\xi = \displaystyle\sum_{j=1}^m \alpha_je_j$ be the boundary point obtained in Theorem \ref{interior},
where $\alpha_j >0$, $m \leq p$ and $e_1, \ldots, e_m$ are orthogonal minimal tripotents.
 Let $h$ be a limit function such that $h(D)$ is weakly closed. Since $f$ is a compact map,
it follows from \cite[Theorem 3.1]{kk} that
 $h(D) \subset \partial D$ (cf.\,\cite[Lemma 6.5]{lie}). By previous remarks,
$h(D)$ is contained in a boundary component $K_u$ of $\overline D$ for some tripotent $u \in \partial D$.

For $n =1,2, \ldots$, pick $y_n$ in the horoball $S_0(\xi,n)$. By $f$-invariance, we have $h(y_n) \in S(\xi, n)$,
which, from Theorem \ref{interior}, is of the form
$$h(y_n) = \sum_{j=1}^m\frac{\sigma_{j}n}{1+\sigma_{j}n}\,e_{j}+ B\left(\sum_{j=1 }^m\sqrt{\frac{\sigma_{j}n}{1+\sigma_{j}n}}\,e_{j}, ~\sum_{j= 1}^m\sqrt{\frac{\sigma_{j}n}{1+\sigma_{j}n}}\,e_{j}\right)^{1/2}(w_n)$$
for some $w_n \in \overline D$. Let $(w_{n_k})$ be a subsequence of $(w_n)$ weakly converging to $w\in \overline D$, say. Then the sequence $(h(y_{n_k}))$ weakly converges to
$$\sum_{j=1}^m  e_{j}+ B\left(\sum_{j=1 }^m e_{j},\, ~\sum_{j= 1}^m e_{j}\right)^{1/2}(w)=
\sum_{j=1}^m  e_{j}+ P_0\left(\sum_{j=1 }^m e_{j}\right)(w) \in \overline K_e$$
where $K_e$ is the boundary component in $\partial D$ containing the tripotent $e = e_1 + \cdots + e_m$.
Since $h(D)$ is weakly closed, we have $\emptyset \neq h(D) \cap \overline K_e \subset K_u \cap
\overline K_e$ and $K_u$ meets either $K_e$ or a boundary component of $\partial K_e$. By the remark
following Lemma \ref{ooc}, the latter is also a boundary component of $\overline D$. It follows that
either $K_u=K_e$ or $K_u$ is a boundary component of $\partial K_e$, that is, $K_u \subset \overline K_e$
which gives $h(D) \subset \overline K_e$.

\end{proof}

\begin{rem} The above result generalises the Denjoy-Wolff theorem for Hilbert balls.
Indeed, for a compact fixed-point free holomorphic self-map $f$ on a Hilbert ball $D$,
we have $\xi =e_1$ and $ K_{e_1}
=\{\xi\}$ in the above
 theorem. Each subsequential limit $h$ of $(f^n)$ is constant \cite[p.1775]{cm1}
 and hence $h(D)=\{\xi\}$. This implies locally uniform convergence of $(f^n)$
 to the constant map taking value $\xi$, by Remark \ref{normal}.
\end{rem}

\begin{exam}\label{final}\rm  Although it is known that the Denjoy-Wolff Theorem fails for a
non-compact holomorphic self-map on a Hilbert ball \cite{s}, we see in the following example that compactness
is not a necessary condition for a Denjoy-Wolff type theorem (see also \cite{cr}). This example also
shows that the image $h(D)$ of a limit function $h$ can be a singleton or a whole boundary component.

Let $D$ be a finite-rank bounded symmetric domain of rank $p$. Pick any nonzero $a\in D$, with spectral decomposition
$$a= \alpha_{1}e_{1}+\dots+\alpha_{p}e_{p} \quad( \|a\|= \alpha_1 \geq \cdots \geq \alpha_p \geq 0).$$
Let $g_a: D\rightarrow D$ be the M\"obius transformation induced by $a$, which is not a compact map if $D$ is
 infinite dimensional. Let $x=\beta_1e_1+\beta_2e_2+\cdots+\beta_p e_p$, where $\beta_1, \beta_2,\ldots,\beta_p \in \mathbb{D}$ so that $x\in D$. By orthogonality, we have
$$ x \bo a =( \beta_1e_1+\beta_2e_2+\cdots+\beta_p e_p) \bo (\alpha_{1}e_{1}+\dots+\alpha_{p}e_{p} )
= \beta_1\alpha_{1}e_{1} \bo e_1+\dots+ \beta_p\alpha_{p}e_{p}\bo e_p$$ and
$$(x\bo a)^n(x) = \beta_1^{n+1}\alpha_{1}^ne_{1} +\cdots+ \beta_p^{n+1}\alpha_{p}^ne_{p} \qquad (n=1,2,\ldots).$$
It follows that
 \begin{eqnarray*}
g_a(x) &=& a + B(a,a)^{1/2}(\mathbf{1}+ x\bo a)^{-1}(x)\\
&=& a+ B(a,a)^{1/2}(\mathbf{1}- x \bo a + (x\bo a)^2 - (x\bo a)^3 + \cdots)(x)\\
&=& a + B(a,a)^{1/2}(\beta_1e_1+\beta_2e_2+\cdots+\beta_p e_p - (\beta_1^2\alpha_{1}e_{1} +\cdots+ \beta_p^2\alpha_{p}e_{p}) + \cdots)\\
&=& a+ B(a,a)^{1/2}[ (1-\beta_1\alpha_1+\beta_1^2\alpha_1^2 + \cdots)\beta_1e_1 + \cdots + (1-\beta_p\alpha_p+\beta_p^2\alpha_p^2 + \cdots)\beta_pe_p)]\\
&=&  a+ B(a,a)^{1/2}\left(\frac{\beta_1e_1}{1+\beta_1\alpha_1} + \cdots +
\frac{\beta_pe_p}{1+\beta_p\alpha_p}\right)\\
&=& \alpha_{1}e_{1}+\dots+\alpha_{p}e_{p}  +  \frac{(1-\alpha_1^2)\beta_1e_1}{1+\beta_1\alpha_1}
+ \cdots + \frac{(1-\alpha_p^2)\beta_p e_p}{1+\beta_p\alpha_p}\\
&=& \frac{\alpha_1 + \beta_1}{1+\alpha_1\beta_1}\,e_1 + \cdots +
\frac{\alpha_p + \beta_p}{1+\alpha_p \beta_p}\,e_p\\
&=& \frak g_{\alpha_1}(\beta_1)e_1 + \cdots + \frak g_{\alpha_p}(\beta_p)e_p
\end{eqnarray*}
where $\frak g_{\alpha_j}$ is the M\"obius transformation on the complex disc $\mathbb{D}$, induced
by $\alpha_j$ for $j =1, \ldots, p$. If $\alpha_j=0$, then $\frak g_{\alpha_j}$ is the identity map.
If $\alpha_j > 0$, then the iterates $(\frak g^n_{\alpha_j})$ converge locally uniformly to
the constant map with value $\alpha_j/|\alpha_j| =1$. Hence the iterates
$$g_a^n(x) = \frak g_{\alpha_1}^n(\beta_1)e_1 + \cdots + \frak g_{\alpha_p}^n(\beta_p)e_p \qquad (n =2,3,\ldots)$$
converge to
$$ e_1 + \gamma_2e_2 + \cdots + \gamma_p e_p, \quad \gamma_j =
\left\{\begin{array}{cc} 1 & (\alpha_j > 0)\\
                         \beta_j & (\alpha_j =0) \end{array}\right.
\quad (j=2, \ldots,p).
$$
In particular, if $\alpha_j > 0$ for all $j$, then by Remark \ref{normal},
 the iterates $(g_a^n)$ converge locally uniformly to
a constant map with value $\xi = e_1 + \cdots + e_p$
which is a maximal tripotent in $\partial D$. On the other hand, if $J =\{j: \alpha_j > 0\}$ is a proper
subset of $\{1, \ldots, p\}$, then
$$\lim_n g_a^n(x) = \sum_{j \in J} e_j  + \sum_{j\notin J} \beta_je_j \in
e  + D_e$$
where $e = \sum_{j \in J} e_j$ is a tripotent in $\partial D$ and $D_e= V_0(e) \cap D$.
It follows that, in this case, the image of every limit function $h$ of $(g_a^n)$ is the whole boundary component
$e  + D_e$ since for any $ e+z \in e+D_e$ with $z\in D_e$ and spectral decomposition $z = \sum_{j\notin J} \beta_j u_j$,
we have $h(\sum_{j\in J} \alpha_j e_j + \sum_{j\notin J} \beta_j u_j) = e+\sum_{j\notin J} \beta_j u_j$.
\end{exam}

\end{document}